\documentclass[10pt]{amsart}
\usepackage{a4, amsmath, bm, amssymb}
\usepackage{setspace}
\usepackage{subfiles}
\usepackage{cite}
\usepackage{enumerate}
\usepackage{url}

\usepackage{a4,amsmath,amssymb,amscd,verbatim,bbm,graphicx,enumerate,tikz}

\definecolor{crimson}{rgb}{0.86, 0.08, 0.24}
\definecolor{bleudefrance}{rgb}{0.19, 0.5, 0.91}
\usepackage[pdftex,colorlinks,linkcolor=black,citecolor=crimson,filecolor=black]{hyperref}

\usepackage{faktor}
\usepackage{dsfont}
\usepackage{amsthm}

\makeatletter
\newtheorem*{rep@theorem}{\rep@title}
\newcommand{\newreptheorem}[2]{%
\newenvironment{rep#1}[1]{%
 \def\rep@title{#2 \ref{##1}}%
 \begin{rep@theorem}}%
 {\end{rep@theorem}}}
\makeatother

\newtheorem{theorem}{Theorem}
\newtheorem{lemma}[theorem]{Lemma}

\newtheorem{proposition}[theorem]{Proposition}
\newtheorem{corollary}[theorem]{Corollary}
\newtheorem{claim}[theorem]{Claim}
\theoremstyle{definition}
\newtheorem{definition}[theorem]{Definition}
\newtheorem{remark}[theorem]{Remark}

\newenvironment{example}[1][Example]{\begin{trivlist}
\item[\hskip \labelsep {\bfseries #1}]}{\end{trivlist}}

\numberwithin{equation}{section}
\numberwithin{theorem}{section}

\begin{document}
  
\normalsize

\title[Discrete group extensions and pressure gap]{Discrete group extensions of countable Markov shifts and pressure gap}

\author{Rhiannon \textsc{Dougall}}

\address{Department of Mathematical Sciences,
Durham University,
Upper Mountjoy,
Durham DH1 3LE}
\email{rhiannon.dougall@durham.ac.uk}

\begin{abstract} 
The main outcome of this note is to extend the work of Stadlbauer (\textit{An extension of Kesten's criterion for amenability to topological Markov chains}, 2013), Coulon--Dal'Bo--Sambusetti (\textit{Growth gap in hyperbolic groups and amenability}, 2018), and Dougall (\textit{Critical exponents of normal subgroups, the spectrum of group extended transfer operators, and Kazhdan distance}, 2019), to the setting of countable Markov shifts with strongly positively recurrent dynamics. We do this using the perspectives given in the work of Coulon--Dougall--Schapira--Tapie (\textit{Twisted Patterson--Sullivan measures and applications to amenability and coverings}, 2025). We conclude the note with context for the twisted measure and estimates for a particular example.
\end{abstract}

\maketitle

\section{Introduction}\label{section:intro}
The topic of group extensions of topological Markov chains with a view to non-abelian, discrete groups, is novel and growing; a snapshot of significant results of the last 15 years is \cite{Stadlbauer13}, \cite{DougallSharp1},\cite{Jae}, \cite{CoulonDalboSambusetti}, \cite{Dougall}, \cite{DougallSharp2}. (The more established topic of abelian group extensions continues to be active, for instance for its relation to Lorentz gas dynamics.) 
In this note is we extend \cite{Stadlbauer13}, and extend \cite{CoulonDalboSambusetti} and \cite{Dougall}, to the setting of a countable Markov shift with strongly positively recurrent dynamics using methods inspired from \cite{CoulonDougallSchapiraTapie}. We conclude this note with some context for the objects in the proof.


Let $\sigma:\Sigma^+\to\Sigma^+$ be a (one-sided) countable Markov shift with alphabet $\mathcal{W}_1$.  We call a word $w\in(\mathcal{W}_1)^k$ \textit{admissible} if it appears as a subword of some $x\in\Sigma^+$. 
Cylinder sets $[w_1\cdots w_k]=\left\{ x=(x_0,x_1,\ldots)\in \Sigma^+ : x_{i-1} = w_{i}, \, i=1,\ldots, k \right\}$ given by admissible words $w_1\cdots w_k$ form a basis for the topology on $\Sigma^+$.
 We always assume that $\sigma$ is mixing. Let $R:\Sigma^+\to \mathbb{R}_+$ with $\log R$ locally H\"{o}lder continuous. We make the stylistic choice to use multiplicative notation, writing $R_n(x) = R(x)\cdots R(\sigma^{n-1}x)$, and that a distinguished letter in $\mathcal{W}_1$ will often be capitalized. 
The Gurevi\v{c} pressure $P(\log R,\sigma)$ is defined as
\begin{equation}\label{equation:pressure}
P(\log R,\sigma) = \limsup_{n\to\infty}\frac{1}{n} \log \sum_{x\in[B]: \sigma^n x = x}R_n(x).
\end{equation}
We henceforth assume that $P(\log R,\sigma)=0$, which may always be achieved by scaling.


We assume that $R$ is strongly positively recurrent in the sense of Sarig \cite{Sarig2001}. This in particular implies the existence of a shift-invariant probability equilibrium state, which we denoted $\mu=\mu_R$. Further details on these constructions are given in Section \ref{section:cms}. We henceforth fix the data of $\sigma,\Sigma^+,\mathcal{W}_1,R,\mu$ and vary the following skew products.



Let $G$ a discrete countable group. A map $\psi:\mathcal{W}_1 \to G$ defines a  \emph{$\psi$-skew product} with phase space $\Sigma\times G$ and dynamical system
$$
T_\psi(x,g) = (\sigma x, \psi(x_0)^{-1}g).
$$
The skew product is isomorphic to a countable Markov shift. We also call $T_{\psi}$ a \emph{discrete group extension}, or $G$-\emph{extension}.
We lift $R$ to $\Sigma^+\times G$ by defining $R(x,g)=R(x)$ (and use the same notation for the function on domain $\Sigma^+$ and on domain $\Sigma^+\times G$). 
For $w\in\mathcal{W}$, $w=w_1\cdots w_n$, define $\underline{\psi}(w)=\psi(w_1)\cdots \psi(w_n)$. Then $$T^n_\psi(x,g) = (\sigma x, \underline{\psi}(x_0\cdots x_{n-1})^{-1}g).$$
The Gurevi\v{c} pressure $P(\log R,T_{\psi})$ is 
$$
\limsup_{n\to\infty}\frac{1}{n}\log \sum_{(x,e)\in[B]\times\left\{e\right\}: T_{\psi}^n (x,e)=(x,e)}R_n(x),
$$
where $e$ is the identity element of $G$. When $T_{\psi}$ is transitive the definition is independent of $B$. We always have $P(\log R,T_{\psi})\le P(\log R,\sigma)=0$. 

\begin{example}[Example: Symmetric random walks.]\label{example:rw}
Given probability $p:S\to (0,1)$ and $S=S^{-1}$ generating $G$, we let $\Sigma^+ = S^\mathbb{N}$, $R((s_0,s_1,\ldots))=p(s_0)$ and $\psi(w)=w$, identifing the alphabet with the group elements. Then $T_{\psi}$ describes sample paths of the random walk driven by $p$. The logarithm of the Gurevi\v{c} pressure coincides with $\limsup_{n\to\infty} p^{\ast n}(e)^{1/n}$. It follows from Kesten's criterion \cite{Kesten} that an amenable group $G$ has $\limsup_{n\to\infty} p^{\ast n}(e)^{1/n}=1$ for every symmetric probability $p$.
\end{example}

\begin{theorem}\label{theorem:transitive}
Assume that $T_{\psi}$ is transitive. If $G$ is non-amenable then we have $P(\log R,T_{\psi})<0$.
\end{theorem}
Theorem \ref{theorem:transitive} was previously obtained by Stadlbauer \cite{Stadlbauer13} under the big images and pre-images property, a combinatorial property of the CMS that implies the Gibbs property for the measure $\mu$. 

For $\psi$ and $G$ fixed we define, for any $H\unlhd G$, $\psi_H(x)=\psi(x)H$ and corresponding skew product $T_{\psi_H}:\Sigma\times G/H \to \Sigma\times G/H $.
Following \cite{Dougall}, \cite{CoulonDalboSambusetti} we say that $\psi$ satisfies the \emph{visibility hypothesis} if there is a finite set $K$ of $G$ so that every $g\in G$ may be expressed $g=k_1 \psi(w_1)\cdots \psi(w_n)k_2$ for $k_1,k_2\in K$ and some admissible word $w_1\cdots w_n$. 

\begin{example}[Example: Non-backtracking walk on the free group quotients.]
Let $F_{a,b}$ be the free group with generators $\left\{a,a^{-1},b,b^{-1}\right\}$. Let
$$
\Sigma^+ = \left\{ x\in\left\{a,a^{-1},b,b^{-1}\right\}^{\mathbb{N}}: x_{i+1}\ne x_i^{-1}\, \forall i\in\mathbb{N} \right\},
$$ 
identifying the alphabet with the group elements. The skew product dynamics on $\Sigma\times F_{a,b}$ with $\psi(x_0)=x_0$ walks along geodesics in $F_{a,b}$. In this case $T_{\psi}$ is not transitive.
\end{example}

We expect the visibility hypothesis to be satisfied in settings where there is a coding of a geometric problem; the visibility hypothesis is used in Coulon--Dal'bo--Sambusetti \cite{CoulonDalboSambusetti} to characterize a uniform gap in critical exponents of certain isometric actions of word hyperbolic groups, and a more restrictive version is used in Dougall \cite{Dougall} to characterize a uniform gap in Gurevi\v{c} entropy for certain geodesic flows.
One of the insights of \cite{CoulonDalboSambusetti} and \cite{Dougall} is that one should understand a family of skew products in $G/H$ in terms of unitary representation of $G$. 
Given $G$ and given any $H\unlhd G$, we have the unitary representation $(\lambda_H,\ell^2(G/H))$ of $G$ defined by $[\lambda_H(g)]f(xH)=f(g^{-1}xH)$. One says that a unitary representation $(\rho,\mathcal{H})$ of $G$ weakly contains the trivial representation, denoted $\mathds{1} \nprec \rho$, if for every $\epsilon>0$ and every finite set $S\subseteq G$ there is a unit vector $v\in \mathcal{H}$ with $\|\rho(s)v-v\|_{\mathcal{H}}\le \epsilon$ for all $s\in S$. (We refer the reader to \cite{bekka} for a comprehensive introduction to these objects.)

\begin{theorem}\label{theorem:LVR}
Assume that $\psi$ satisfies the visibility hypothesis.
Let $\mathcal{F}$ be a family of normal subgroups of $G$ and assume that $T_{\psi_H}$ is transitive for each $H\in\mathcal{F}$. We have
$$
\mathds{1} \nprec \bigoplus_{H\in\mathcal{F}} \lambda_H \implies \sup\left\{P(\log R,T_{\psi_H}): H\in\mathcal{F}\right\}<0.
$$
\end{theorem}

The arguments are structured as follows.
In section \ref{section:growthseries} we relate Gurevi\v{c} pressure problems to the convergence of power series associated to elements of $\ell^2(G/H)$. 
In section \ref{section:proof} we discuss a notion of twisted measures and their role in ultimate main theorems.
In section \ref{section:existence} we give a (lengthy) existence result for the twisted measures thus complete the proof of the theorems.

We conclude this note with some further context for the twisted measures. In particular we show that the twisted measure of Definition \ref{def:twistedupsilon}, in the context of a free group extension with random walk data as in Example \ref{example:rw}, is a quantitatively different from the equilibrium state. 


\section{Countable Markov shifts and strong positive recurrence}\label{section:cms}
Let $x \in \mathbb{N}^{\mathbb{N}_0}$; that is, $x$ is a one-sided infinite sequence in the countable set $\mathbb{N}_0$ --- we will often write $\mathcal{W}_1$ in place of $\mathbb{N}$, and describe $\mathcal{W}_1$ as the alphabet of the CMS. As is common usage, we use $x_i$ to denote the $i$th element in the sequence. In the theory of Markov shifts and denote an element $x\in\mathbb{N}^{\mathbb{N}_0}$ by $x = x_0x_1\ldots$.
In order to define a CMS it is useful to make reference to a \emph{transition matrix} $\tau: \mathcal{W}_1\times \mathcal{W}_1\to \left\{ 0,1\right\}$. The \emph{(one-sided) countable Markov shift (with transition matrix $\tau$)} is 
$$
\Sigma^+ = \left\{ x\in \mathcal{W}_1^{\mathbb{N}_0} : \tau(x_i,x_{i+1})=1\, \forall i\in\mathbb{N}_0\right\}.
$$
We use $\sigma$ to denote the left shift $\sigma : \Sigma^+\to \Sigma^+$, $(\sigma x)_i = x_{i+1}$. We always assume that the dynamics are transitive for $\sigma :\Sigma^+ \to \Sigma^+$.
We equip $\Sigma^+$ with the product topology, which can be metrized: write $x\wedge y = \min ( k+1 : x_i=y_i, 0\le i \le k )$, and $d(x,y) = 2^{- x\wedge y}$. We use some convenient notation for a basis of open balls,
$$
[w_1\cdots w_n] = \left\{ x\in \Sigma : x_j=w_j, j=0,\ldots, n \right\}
$$
where $w_j,w_{j+1}\in\mathcal{W}_1$, $\tau(w_j, w_{j+1})=1$, $j=0,\ldots , n-1$. We say that such a $w=w_1\cdots w_n$ is \emph{admissible}, and $|w|=n$. For $B,b\in\mathcal{W}_1$ we write
$$
\mathcal{W}^{B,b}_n = \left\{ B u_1\cdots u_{n-2} b : u_j,u_{j+1}\in\mathcal{W}_1, \tau(u_j, u_{j+1})=1, j=1,\ldots , n-3 \right\}
$$
and
$$
\mathcal{W}^{B,b} = \left\{ B u_1\cdots u_{n} b : u_j,u_{j+1}\in\mathcal{W}_1, \tau(u_j, u_{j+1})=1, j=1,\ldots , n; n\in\mathbb{N} \right\}
$$

Let $R:\Sigma^+ \to \mathbb{R}_+$ be a positive function. We say that $\log R$ is locally H\"{o}lder continuous if there is $\theta<1$ for which
$$
\mathrm{Var}_n(\log R) = \sup \left\{ |\log{ R}(x)-\log R(y)| : x\wedge y = n \right\}
$$
has 
\begin{equation}\label{equation:lip}
\mathrm{Var}_n(\log R)\le C_{(\ref{equation:lip})}(\log R) \theta^n
\end{equation}
for $n>1$. It should be noted that $\log R$ can be unbounded on any cylinder when $\Sigma^+$ is non-compact.

We write $C_{(\ref{localholder})}(Bb)$ for the constant with
\begin{equation}\label{localholder}
C_{(\ref{localholder})}(Bb)^{-1}\le \frac{R_{k+1}(uBby)}{{R_{k+1}(uBbz)}} \le C_{(\ref{localholder})}(Bb),
\end{equation}
for any admissible $uBb$, $|u|=k$, $|b|=|B|=1$, and $y,z\in\sigma [b]$. 

\subsection{Strong positive recurrence by discriminants}\label{sprdef}
Fix $a\in\mathcal{W}_1$. Following \cite{Sarig} we say that $R$ is positive recurrent if there is a constant $M_a$ with
$$
\sum_{x: \sigma^n x = x, x_0=a}R_n(x) \in [M_a^{-1}\exp nP(\log R,\sigma),M_a\exp nP(\log R,\sigma)]
$$
for all $n\in\mathbb{N}$.
The definition of strong positive recurrence was introduced by Sarig \cite{Sarig2001} in terms of discriminants. We switch to additive notation -- i.e.\ $R=\exp r$ with $r$ locally H\"{o}lder continuous -- for this discussion. Following \cite{Sarig2001}, write $\varphi_a = \mathds{1}_{[a]}(x) \inf\left\{ n\ge 1: \sigma^n x \in [a]\right\}$,
\begin{equation}\label{sarig1}
Z_n(r) = \sum_{\sigma^n x  = x} (\exp r)_n(x)\mathds{1}_{[a]}(x) \; ;\; Z^*_n(r) = \sum_{\sigma^n x  = x} (\exp r)_n(x)\mathds{1}_{[\varphi_a=n]}(x)
\end{equation}
\begin{equation}\label{sarig2}
P(r):=P(r,\sigma) = \limsup_{n\to\infty} \frac{1}{n}\log Z_n(r) \; ;\; -p^*:=\limsup_{n\to\infty}\frac{1}{n}\log Z^*_n( r) 
\end{equation}
We write $P(\overline{ r})$ for the pressure in the system induced on returning to $a$. (So $\overline{ r} = \exp (\sum_{k=0}^{\varphi_a -1} r\circ \sigma^n )$.)
 The Discriminant Theorem (Theorem 2 of  \cite{Sarig2001}) states that $P(\overline{r+p})=0$ has a unique solution $p=-P( r)$ if $r$ is not transient. The first return series controls the range of $p$ for which $P(\overline{r+p})<\infty$, indeed $\sup\left\{ p : P(\overline{r+p})<\infty \right\} = p^*$. The $a$-discriminant of $r$ is defined to be $\Delta_a[r] = \sup \left\{ P(\overline{r+p}) : p<p^*\right\}$. Non-transience implies that $\Delta_a[r]\ge 0$, and indeed $P(\overline{r-P(r)})=0$. Following Sarig \cite{Sarig2001}, one says that $r$ is \emph{strongly positively recurrent} if $\Delta_a[r]>0$.
(At $\Delta_a[r]=0$ the behaviour can be either positive recurrence or null recurrence.)
 
The main thing for us to note is Proposition 3 that $P(\overline{r+p})$ is strictly increasing in $(-\infty, p^*]$. Now
if $P(r)\ne -p^*$ then there is some $-P(r)<t<p^*$ with $0<P(\overline{r+t})<\infty$ from which $\Delta_a[r]>0$. Conversely, if $P(r)= -p^*$ then $\Delta_a[r]=0$. In conclusion, under strong positive we have
\begin{equation}\label{pstar}
P(r)\ne -p^*.
\end{equation} 

\subsection{Return series and first return series}

Let $R :\Sigma^+\to\mathbb{R}_+$ with $\Sigma^+$ some CMS and $\log R$ locally H\"{o}lder continuous. 
To begin with we do not even ask that $R$ is recurrent --- in this way the discussion applies to both $\sigma:\Sigma\to \Sigma$ and to the skew product $T_{\psi}:\Sigma\times G\to \Sigma \times G$, upon realising an isomorphism with a CMS.
For a letter $B\in \mathcal{W}_1$, the \emph{($B$-conditioned) (periodic) return series} is
\begin{equation}\label{Bcpts}
t\mapsto \sum_{n=1}^\infty t^{-n} \sum_{x: \sigma^n x = x, x_0=B}R_n(x),
\end{equation}
where $R_n(z) = R(z)R(\sigma z) \cdots R(\sigma^{n-1}z)$.
Using local H\"{o}lder continuity \ref{Bcpts} has the same abscissa of convergence as the \emph{($B$-conditioned) return series}
\begin{equation}\label{Bcts}
t\mapsto \sum_{n=1}^\infty t^{-n} \sum_{b: \tau(b,B)=1}\sum_{w\in\mathcal{W}^{B,b}_n} R_n(wy),
\end{equation}
with fixed \emph{initial condition} $y\in[B]$. And indeed if $\sigma$ is transitive then the abscissa of convergence is equal to $\exp P(\log R,\sigma)$.

The \emph{periodic $B$ first return series},
\begin{equation}\label{pfrs}
t\mapsto \sum_{n=1}^\infty t^{-n} \sum_{x : \sigma ^nx =x, x_0= B, x_i\ne B, 0<i<n}R_n(x),
\end{equation}
is within a constant multiple of 
\begin{equation}\label{frs}
\sum_{n=1}^\infty t^{-n} \sum_{b: \tau(b,B)=1}\sum_{w\in\mathcal{W}^{B,b}_n: w_i\ne B, i>1} R_n(wy).
\end{equation}
We say that $R$ has a \emph{growth gap} if (\ref{pfrs}), or equivalently (\ref{frs}), converges for $t=\exp P(\log R,\sigma)-\epsilon$, for some $\epsilon>0$. We let $\gamma(\mathrm{SPR})$ be the infimum over all such $t=\exp P(\log R,\sigma)-\epsilon$.

\begin{lemma}
If $R$ is strongly positively recurrent then it has a growth gap.
\end{lemma}
\begin{proof}
The series $Z_n^*(r)$ of (\ref{sarig1}) are precisely the coefficients in the periodic $a$ first returns series of (\ref{pfrs}). As stated, these converge precisely when (\ref{frs}) converges, and hence for $t$ arbitarily close to $\exp (-p^*)$. Since $\exp P(r)> \exp (-p^*)$ there must be $t<\exp P(r)$ for which we have convergence, whence a growth gap.
\end{proof}

We find it more convenient to work with returns conditioned on two-letters.
Fix $A,a\in\mathcal{W}_1$ with $aA$ admissible. Define the \emph{($A,a$ conditioned) return series} by
\begin{equation}\label{Aacts}
\zeta^{A,a}[t] = \sum_{n=1}^\infty t^{-n} \sum_{w\in\mathcal{W}^{A,a}_n} R_n(wx),
\end{equation}
and the \emph{($A,a$ conditioned) first return series} 
as 
\begin{equation}\label{Aafrs}
\eta^{A,a}[t]:=\sum_{n=1}^\infty t^{-n} \sum_{w\in\mathcal{A}^\prime_n}e^{R_n(wx)}.
\end{equation}
where $\mathcal{A}^\prime_n = \left\{ Awa : w\in\mathcal{W}_{n-2}, w_iw_{i+1} \ne aA, i=1,\ldots, n-3 \right\}$. In each case we assume that the \emph{initial condition} $x$ satisfies $x\in [A]$.
We also make reference to the approximations
$$
\zeta^{A,a}_{\le N}[t] = \sum_{n=1}^N t^{-n} \sum_{w\in\mathcal{W}^{A,a}_n} R_n(wx)
$$

Since conditioning on more can only reduce the value of the series, we immediately get:
\begin{corollary}\label{cor:finite}
For each $t>\gamma(\mathrm{SPR})$ we have $\eta^{A,a}[t],\zeta_{\le N}^{A,a}[t]<\infty$.
\end{corollary}
On the other hand it will be convenient to check that we don't lose too much.
\begin{lemma}\label{lemma:growthgap}
If $R$ has a growth gap then 
$$
\inf \left\{ t>0 : \zeta^{A,a}[t]<\infty \right\} = \exp P(\log R,\sigma).
$$
If $R$ is recurrent then $\zeta^{A,a}[\exp P(\log R,\sigma)] = \infty$.
\end{lemma}
\begin{proof}
Notice that
\begin{align*}
&\sum_{n=1}^\infty t^{-n} \sum_{w\in\mathcal{W}^{A,a}_n} R_n(wx) \sum_{m=1}^\infty t^{-m} \sum_{b: \tau(b,A)=1}\sum_{u\in\mathcal{W}^{A,b}_m: u_i\ne A, i>1} R_m(ux)
\\
&\ge C_{(\ref{localholder})}(Aa) \sum_{n=1}^\infty t^{-n} \sum_{w\in\mathcal{W}^{A,a}_n} \sum_{m=1}^\infty t^{-m} \sum_{b: \tau(b,A)=1}\sum_{u\in\mathcal{W}^{A,b}_m: u_i\ne A, i>1} R_{n+m}(wux)
\\
&= C_{(\ref{localholder})}(Aa)\sum_{n=1}^\infty\sum_{m=1}^\infty t^{-m} t^{-n} \sum_{w\in\mathcal{W}^{A,a}_n} \sum_{b: \tau(b,A)=1}\sum_{u\in\mathcal{W}^{A,b}_m: u_i\ne A, i>1} R_{n+m}(wux)
\\
&= C_{(\ref{localholder})}(Aa)\sum_{n=1}^\infty\sum_{m=1}^\infty t^{-m} t^{-n} \sum_{b: \tau(b,A)=1}\sum_{v\in\mathcal{W}^{A,b}_{n+m}} R_{n+m}(vx)
\\
&\ge C_{(\ref{localholder})}(Aa)\sum_{k=1}^\infty t^{-k} \sum_{b: \tau(b,A)=1}\sum_{v\in\mathcal{W}^{A,b}_{k}} R_{k}(vx).
\end{align*}
\end{proof}

\subsection{Equilibrium states and conformal measures}
Assume that $P(\log R,\sigma ) = 0$. Positive recurrence guarantees that the transfer operator $L$ has an eigenfunction $Lh = h$, with $h$ a positive locally H\"{o}lder continuous and eigenmeasure $L^* \nu =  \nu$;  see \cite{Sarig}. The equilibrium state $d\mu = hd\nu$ is finite, $\sigma$ invariant and ergodic. 

When $\Sigma^+$ is compact $\mu$ has the \emph{Gibbs property} (see \cite{PP}): there is a constant $C_{\mathrm{Gibbs}}>0$ with
\begin{equation}\label{equation:gibbs}
C_{\mathrm{Gibbs}}^{-1}R_n(wy)\le \mu([w])\le C_{\mathrm{Gibbs}}R_n(wy)
\end{equation}
for any admissible $w$ of length $n$ and $y\in\sigma^n [w]$.
In general we cannot expect to have the Gibbs property in the CMS setting but we are able to make use of the conformal property for $\nu$. Namely, for $w\in \mathcal{W}^{d,B}$, with $Bb$ admissible, we have
\begin{equation}\label{conformal}
\int R_n^{-1}\mathds{1}_{[wb]} d\nu = \nu([b]). 
\end{equation}
It follows that one has a local version of a Gibbs inequality for $\nu$. Namely, for each $B,b\in\mathcal{W}_1$ with $Bb$ admissible we have
\begin{equation}\label{localgibbs}
C_{(\ref{localholder})}(Bb)^{-1}  R_n(wbx)\le \nu([wb])\le C_{(\ref{localholder})}(Bb) R_n(wbx),
\end{equation}
for and $w\in\mathcal{W}^{d,B}_n$ and any $x\in \sigma [b]$. 

We also have the following estimate for $\mu$. For any $\xi\in [b]$, and $w\in \mathcal{W}^{d,B}_n$,
\begin{equation}\label{localgibbsmu}
\int R^{-1}_n \mathds{1}_{[wb]} d\mu \ge \nu([b]) (h(w\xi) + C_{(\ref{equation:lip})}(h)\theta^n),
\end{equation} 
since
\begin{align*}
\int R^{-1}_n \mathds{1}_{[wb]} d\mu &= \int R^{-1}_n \mathds{1}_{[wb]} hd\nu 
\\
&=\int L^n(R^{-1}_n \mathds{1}_{[wb]} h) d\nu 
\\
&= \int_{[b]} h(wz) d\nu (z)
\end{align*}
and since $|h(wz)-h(w\xi)| \le C_{(\ref{equation:lip})}(h) \theta^n$.


\section{Growth series, matrix coefficients, divergence}\label{section:growthseries}

Recall that we fixed the data of the countable Markov shift and log-H\"{o}lder function $R$ with $P(\log R,\sigma)=0$. We let $G$ be a countable group and let $\psi:\mathcal{W}_1\to G$. 
Recall that we write $\mathcal{W}^{A,a}_n$ for admissible words $w=w_1\cdots w_n$ of length $n$ with first letter $w_1=A$ and last letter $w_n=a$. 
\begin{definition}\label{def:phi}
Let $A,a\in\mathcal{W}_1$ and $x\in \sigma [a]$. Let $c:\mathbb{N}\to [0,\infty)$ with subexponential growth. Let $t>\gamma(\mathrm{SPR})$. Define the partial sums $\phi^{A,a}_{c;\le N}[t]: G\to [0,\infty)$ by
$$
\phi^{A,a}_{c;\le N}[t](g) = \sum_{n=1}^N t^{-n}c_n \sum_{w\in\mathcal{W}_n^{A,a}: \underline{\psi}(w)=g}R_n(wx).
$$
For any such $t$ for which $\sup_{N} \phi^{A,a}_{c;\le N}[t](g)<\infty$ for each $g\in G$, we denote $\phi_c^{A,a}[t]:G\to [0,\infty)$
$$
\phi^{A,a}_c[t](g) = \sum_{n=1}^\infty t^{-n} \sum_{w\in\mathcal{W}_n^{A,a}: \underline{\psi}(w)=g}R_n(wx).
$$

In the case that $c$ is identically $1$ we abbreviate to $\phi^{A,a}_{\le N}[t](g)$ and $\phi^{A,a}[t]$ respectively.
If $\Sigma^+$ is not compact we insist that $aA$ is admissible, and $x\in [A]$.
\end{definition}
The well-definedness of the Definition \ref{def:phi} follows from Corollary \ref{cor:finite}.
We explain the effect of the choice of $A,a$. We also explain the effect of a permuation by multiplying by an element of $G$.
For elements $v,w\in \mathbb{R}^G$ we say that $v\le w$ if $v(g)\le w(g)$ for every $g\in G$. We use the notation of convolution: in particular for a function $\phi:G\to \mathbb{R}$, $\delta_h\ast \phi$ denotes the $\phi(g)=\phi(h^{-1} g)$, $g\in G$.
\begin{lemma}\label{lemma:changeletter}
For each subexponentially growing $c$ and each $B\in\mathcal{W}_1$ there are constants $D_{(\ref{equation:changeletter})}(c,A,B)$, $K_B$, and group elements $h_B$,$g_B$ with
\begin{equation}\label{equation:changeletter}
D_{(\ref{equation:changeletter})}(c,A,B)^{-1}\, \delta_{h_B}\ast \phi^{A,a}_{c;\le N-K_B}[t] \le \phi^{B,a}_{c;\le N}[t]\le  D_{(\ref{equation:changeletter})}(c,A,B) \,\delta_{g_B} \ast \phi^{A,a}_{c;\le N+K_B}[t]. 
\end{equation}
If $T_{\psi}$ is transitive then given $g\in G$ there are constants $D_{(\ref{equation:changeletter2})}(c,A,B,g),L_g$, such that
\begin{equation}\label{equation:changeletter2}
D_{(\ref{equation:changeletter2})}(c,A,B,g)^{-1}\, \delta_{g}\ast \phi^{B,a}_{c;\le N-L_g}[t] \le \phi^{A,a}_{c;\le N}[t]\le D_{(\ref{equation:changeletter2})}(c,A,B,g) \,\delta_{g} \ast \phi^{B,a}_{c;\le N+L_g}[t]. 
\end{equation}
\end{lemma}
\begin{proof}
We need to upper bound $\phi^{B,a}$ in terms of $\phi^{A,a}$. Choose $u_B$ with $Au_BB$ admissible, in this way any $v\in\mathcal{W}^{B,a}_n$ has $Au_B v\in\mathcal{W}^{A,a}_{n+|u_B|+1}$. Let $C_{(\ref{localholder})}(u_BB)$ denote $C_{(\ref{localholder})}(dB)$ with $d$ the final letter of $u_B$. Then, 
 $$
 R_{|Au_B|+n}(Au_Bvx) = R_{|Au_B|}(Au_Bvx)R_{n}(vx) \ge C_{(\ref{localholder})}(u_BB)^{-1} R_{|Au_B|}(Au_Bz)R_{n}(vx)
 $$
 for some fixed $z\in [B]$. Let $C$ be the constant with $c_n\le C c_{n-|u_B|}$. 
Let $\lambda$ denote the left regular representation of $G$ on $\ell^2(G)$. We deduce that
\begin{align*}
& \sum_{n=1}^Nt^{-n}c_n\sum_{v\in\mathcal{W}_n^{B,a}}
  R_{n}(vx) \langle \lambda(\underline{\psi}(v))f,\delta_h\rangle_{\ell^2(G)}
 \le
   \frac{C_{(\ref{localholder})}(u_BB)}{R_{|u_B|+1}(Au_BBz)}\times
   \\
   &\times
\sum_{n=1}^Nt^{-n}c_n\sum_{v\in\mathcal{W}_n^{B,a}}
  R_{|u_B|+1+n}(Au_Bvx)\langle \lambda(\underline{\psi}(Au_B))^{-1}\lambda(\underline{\psi}(Au_B)\underline{\psi}(v))f,\delta_h\rangle_{\ell^2(G)}
 \\
&\le
   C\frac{C_{(\ref{localholder})}(u_BB)}{R_{|u_B|+1}(Au_BBz)}t^{|u_B|+1}
\langle \phi_{c;\le N+|u_B|}[t]\ast f,\lambda(\underline{\psi}(Au_B))\delta_h\rangle_{\ell^2(G)}
\end{align*}
The first statement follows. The second statement is similar.
\end{proof}
We continue the exposition for a fixed $A,a,x$, and use the shorter notation $\phi[t]=\phi^{A,a}[t]$ (and $\phi_c[t]=\phi_c^{A,a}[t]$). The following is immediate from Corollary \ref{cor:finite}.
\begin{lemma}\label{lemma:ell1}
For $t>\gamma(\mathrm{SPR})$ we have that $\phi_{\le N}[t]\in \ell^1(G)$. 
\end{lemma}
\begin{definition}\label{def:integrability1}
Denote
$$
\gamma(\delta_{e_{G/H}}) = \inf\left\{t>\gamma(\mathrm{SPR}) :  \sup_{N\in\mathbb{N}}  \sum_{h\in H} \phi_{\le N}[t](h)<\infty \right\}.
$$
\end{definition}
Notice that $\sum_{h\in H} \phi_{\le N}[t](h) = \overline{\phi}[t]$ for $\overline{\phi}[t]$ the series in Definition \ref{def:phi} associated to the group $G/H$ 
\begin{proposition}\label{prop:gurevic} If $P_{\mathrm{Gur}}(R,T_{\psi_H})>\log \gamma(\mathrm{SPR})$ then
$\log \gamma(\delta_{e_{G/H}})=P_{\mathrm{Gur}}(R,T_{\psi_H})$.
\end{proposition}
\begin{proof}
The assumption $\exp P(\log R, T_{\psi_H})>\gamma(\mathrm{SPR})$ implies that $R:\Sigma^+\times G\to\mathbb{R}_+$ (viewed as a CMS) has a growth gap. In particular by Lemma \ref{lemma:growthgap}, 
$$
\inf \left\{ t>0 : \zeta^{(A,e_{G/H}),(a,e_{G/H})}[t]<\infty \right\} = \exp P(\log R, T_{\psi}).
$$
All that remains is to note that $ \zeta^{(A,e_{G/H}),(a,e_{G/H})}[t]$ coincides with $\sum_{h\in H} \phi[t](h)$.
\end{proof}

Let $(\rho,\mathcal{H})$ be a unitary representation of $G$ obtained as a (finite or) countable direct sum, indexed by $H\in\mathcal{F}$, of quotient representations $(\lambda_H,\ell^2(G/H))$. We make use of the cone of non-negative functions 
$$
\ell^2_+(G/H)=\left\{ f\in\ell^2(G/H): \forall g\in G\, f(g)\ge 0\right\};
$$ 
and denote $\mathcal{H}_+$ the corresponding cone in $\mathcal{H}$. For any function $h\in \ell^1(G)$ we have that the convolution $h\ast_\rho f:\sum_{g\in G} h(g)\rho(g) f$. 
It follows from Lemma \ref{lemma:ell1} that each $\phi_{\le N}[t]\ast_\rho f$ is well-defined for $t>\gamma(\mathrm{SPR})$.

\begin{definition}\label{def:integrability}
For $f\in\mathcal{H}_+$ denote
$$
\gamma(\phi\ast_\rho f) = \inf\left\{t>\gamma(\mathrm{SPR}) :  \sup_{N\in\mathbb{N}}\| \phi_{\le N}[t]\ast_\rho f\|_{\mathcal{H}}  <\infty \right\}.
$$
\end{definition}
We refer to $\gamma(\phi\ast_\rho f)$ as the \emph{decay exponent for $f\in\mathcal{H}$} (for context see Section \ref{section:musings}).
It is easily seen that $\gamma(\delta_{e_{G/H}}) \le \gamma(\phi \ast_\rho  f)$ for any $H\in\mathcal{F}$. 
\begin{proposition}\label{prop:supdelta}
There is $v\in\mathcal{H}_+$ with
$$
\sup_{0\ne f\in\mathcal{H}_+} \gamma(\phi\ast_\rho f)=  \gamma(\phi\ast v).
$$
In particular, there is $v\in\mathcal{H}_+$ with
$$
\log \gamma(\phi\ast_\rho v)\ge \sup_{H\in\mathcal{F}} P_{\mathrm{Gur}}(R,T_{\psi_H}).
$$
\end{proposition}

\begin{proof}
Let $f_k$ be unit vectors approaching the supremum. Set $v=\sum_{k=1}^\infty 2^{-k}f_k$. By positivity 
$$
\langle \phi_{\le N}[t]\ast_\rho v,\phi_{\le N}[t]\ast_\rho v\rangle_{\mathcal{H}} \ge 2^{-2k}\langle \phi_{\le N}[t]\ast_\rho f_k,\phi[t]_{\le N}\ast_\rho f_k\rangle_{\mathcal{H}}. 
$$
Hence, for any $t>\gamma(\phi\ast_\rho v)$ we have a finite upper bound for $\langle \phi_{\le N}[t]\ast_\rho f_k,\phi_{\le N}[t]\ast_\rho f_k\rangle_{\mathcal{H}}$, giving $t\ge \gamma(\phi\ast_\rho f_k)$. Taking an infimum over such $t$ yields $\gamma(\phi\ast_\rho v)\ge \gamma(\phi\ast_\rho f_k)$ for every $k$.
\end{proof}

For $t>\gamma(\delta_e)$, we have that the weak limit $\phi_c[t]= \lim_{N\to\infty} \phi_{c;\le N}[t]$ is well-defined as an element of $\mathbb{R}^G$. The element $\phi_c[t]\ast_\rho f$ is not \textit{a priori} defined, but we can fix an enumeration $g_1,g_2,\ldots$ of $G$ and define $\phi_c[t]\ast_\rho f$ as the limit of $\sum_{i=1}^m \phi[t](g_i) \rho(g_i) f$ should it converge.
Recall that in a Hilbert space a sequence of vectors $v_n$ is said to strongly converge to $v$ if the norms of $v_n-v$ tend to $0$ as $n\to\infty$. 
\begin{lemma}\label{lemma:strongconvergence}
Let $f\in\mathcal{H}_+$ and let $t>\sup_{H\in\mathcal{F}}\gamma(\delta_{e_{G/H}})$. If $\sup_{N\in\mathbb{N}}\|\phi_{c;\le N}[t]\ast_\rho f\|_{\mathcal{H}}<\infty$ then $\phi_{c;\le N}[t]\ast_\rho f$ strongly converges as $N\to\infty$. Moreover the limit is $\phi_c[t]\ast_\rho f$ assuming the $\psi_H$ extensions are transitive. 
\end{lemma}
\begin{proof}
Enumerate an orthonormal basis $(\mathrm{e}_i)_{i\in\mathbb{N}}$ of $\mathcal{H}$, by $\delta_{e_H}$, $H\in\mathcal{F}$, and thier images under $\rho(g)$, $g\in G$. The basis elements belong to $\mathcal{H}_+$. For any $N\in\mathbb{N}$ we have
$$
0\le \langle X_{N},\mathrm{e}_i\rangle_{\mathcal{H}}\le \langle X_{N+1},\mathrm{e}_i\rangle_{\mathcal{H}} <\infty.
$$
We can use this to show that $(X_N)_{N\in\mathbb{N}}$ are Cauchy in $\mathcal{H}$. Note that to compute the norm $\|X_N-X_M\|_{\mathcal{H}}$ it is sufficient to compute inner products with vectors in $\mathcal{H}_+$.
An arbitrary $v\in\mathcal{H}_+$ may be written $v=\sum_{q=1}^\infty a_{q} \mathrm{e}_{q}$.
Choose Q with
$$
\left\|
 \sum_{n=Q}^\infty a_{n}\mathrm{e}_{n}
\right\|_{\mathcal{H}}
\le \epsilon/2.
$$
Then we have
$$
\langle X_N - X_M, v\rangle_{\mathcal{H}} \le \sum_{q\le Q}a_q
\langle 
X_N - X_M, \mathrm{e}_{q}
\rangle_{\mathcal{H}}
+
\langle 
X_N - X_M, \sum_{q> Q}a_q\mathrm{e}_{q}
\rangle_{\mathcal{H}}.
$$
The final term is bounded by $\epsilon\sup_{n\in\mathbb{N}} \|X_n \|_{\mathcal{H}}$.
Given $Q$ is fixed, for $N,M$ large enough we have
$$
\sum_{q\le Q}a_q
\langle 
X_N - X_M, \mathrm{e}_{q}
\rangle_{\mathcal{H}}\le \epsilon.
$$ 
It follows that $X_N$ is indeed a Cauchy sequence and hence converges.

Denote $Y_m=\sum_{i=1}^m \phi[t](g_i) \rho(g_i) f$. For a fixed $q$, for all $m$ large enough we have
$$
\langle Y_m - X_\infty , \mathrm{e}_q\rangle_{\mathcal{H}} \le \limsup_{N\to\infty} \langle Y_m - X_N , \mathrm{e}_q\rangle_{\mathcal{H}} = 0,
$$
using the convergence of each $\sum_{h\in H}\delta_h\ast \phi_{\le N}[t](e)$ and (from the second part of \ref{equation:changeletter}) each $\delta_{g}\ast \sum_{h\in H}\delta_h\ast \phi_{\le N}[t](e)$.
It follows that $Y_m$ converges to $X_\infty$.
\end{proof}
\begin{remark}
In fact we only need the first statement in Lemma \ref{lemma:strongconvergence}, that $\phi_{c;\le N}[t]\ast_\rho f$ converge. Nevertheless, the notation $\phi_{c}[t]\ast_\rho f$ is preferable for this limit, and since it is suggestive of a convolution we include the demonstration of its validity.
\end{remark}

\begin{definition}[Matrix coefficients]\label{def:upsilon} Given $c$ and given a sequence $t_q> \gamma(\phi\ast_\rho f)$ with $\inf_{q\in \mathbb{N}} t_q = \gamma(\phi\ast_\rho f)$ we define
$$
\Upsilon(g) = \inf_{q\in\mathbb{N}}\frac{\langle \rho(g) \phi_c[t_q]\ast_\rho f, \phi_c[t_q]\ast_\rho f \rangle_{\mathcal{H}}}{\langle \phi_c[t_q]\ast_\rho f, \phi_c[t_q]\ast_\rho f\rangle_{\mathcal{H}}}.
$$
\end{definition}
We can always pass to a subsequence for which the infimum in Definition \ref{def:upsilon} is a limit for all $g\in G$. We suppress the dependency of $\Upsilon$ on $c$ and on $t_q$ because we shall now make a choice for these.
As is typical in the Patterson--Sullivan theory (for instance \cite{Patterson}, \cite{Roblin}, \cite{urbanski}), we should prefer that the series corresponding to $\phi[t]\ast_\rho f$ diverge at the growth parameters, and this is the benefit of including $c$. 
First we observe that subexponentially growing functions don't change the abscissa of convergence.
\begin{lemma}\label{lemma:subexponential}
For all $t>\gamma(\phi\ast_\rho f)$ we have
$$
\langle \phi_c[t]\ast_\rho f,\phi_c[t]\ast_\rho f\rangle_{\mathcal{H}}<\infty
$$
\end{lemma}

\begin{proof}
Since $c$ is grows subexponentially, for each $\delta<1$ there is $C_\delta$ with $c_n\le C_{\delta}\delta^{-n}$ for all $n\in\mathbb{N}$. By positivity,
$$
\langle \phi_{c;\le N}[t]\ast_\rho f, \phi_{c;\le N}[t]\ast_\rho f\rangle_{\mathcal{H}}
\le  C_{\delta}\langle \phi_{\le N}[\delta^{-1} t]\ast_\rho f, \phi_{\le N}[\delta^{-1} t]\ast_\rho f\rangle_{\mathcal{H}}.
$$
Taking the supremum over $N$ concludes the proof.
\end{proof}

We say that $c$ is \emph{slowly varying} if for each $k\in\mathbb{N}$,
$$
\frac{c_{n+k}}{c_n}\to 1 \,\mathrm{as}\,n\to\infty.
$$
\begin{lemma}\label{lemma:cphiastastphi}
There is a slowly varying $c$ with
$$
\langle \phi_c[t]\ast_\rho f,\phi_c[t]\ast_\rho f\rangle_{\mathcal{H}}<\infty  \; \mathrm{for}\; t>\gamma(\phi\ast_\rho f)
$$
and
$$
\langle \phi_c[t]\ast_\rho f,\phi_c[t]\ast_\rho f\rangle_{\mathcal{H}} \to\infty \; \mathrm{as}\; t\to\gamma(\phi\ast_\rho f).
$$ 
\end{lemma}

\begin{proof}
The first part follows from Lemma \ref{lemma:subexponential}. For the second we check that 
$\gamma:=\gamma(\phi\ast_\rho f)$ is the abscissa of convergence of the real series
$$
\pi[t] = \sum_{N=1}^\infty t^{-N} \sum_{m+n=N} \sum_{V\in\mathcal{W}^{A,a}_m}\sum_{v\in\mathcal{W}_n^{A,a}}R_m(Vx)R_n(vx)\langle \rho(\underline{\psi}(V))^{-1}\rho(\underline{\psi}(v))f,f\rangle_{\mathcal{H}}.
$$
Write 
$$
\pi_{\le N}[t] = \sum_{K\le N} t^{-K} \sum_{m+n=K} \sum_{V\in\mathcal{W}^{A,a}_m}\sum_{v\in\mathcal{W}_n^{A,a}}R_m(Vx)R_n(vx)\langle \rho(\underline{\psi}(V))^{-1}\rho(\underline{\psi}(v))f,f\rangle_{\mathcal{H}}
$$
Set $\phi_n[t](g) = \sum_{w\in\mathcal{W}^{A,a}_n: \underline{\psi}(w)=g}R_n(wx)$. It is easy to see that
\begin{align*}
&\pi_{\le N}[t]  \le \sum_{n=1}^{N-1}\sum_{m=1}^{N-1} \langle \phi_n[t]\ast_\rho f, \phi_m[t]\ast_\rho f\rangle_{\mathcal{H}}
=\langle  \phi_{\le N}[t]\ast_\rho f, \phi_{\le N}[t]\ast_\rho f\rangle_{\mathcal{H}}
\end{align*}
On the other hand
\begin{align*}
&\langle \phi_{\le N}[t]\ast_\rho f,\phi_{\le N}[t]\ast_\rho f\rangle_{\mathcal{H}}= 
\sum_{k\le 2N} \sum_{\substack{ m+n=k : \\ m\le N,\, n\le N}}\langle \phi_{n}[t]\ast_\rho f,\phi_{m}[t]\ast_\rho f\rangle_{\mathcal{H}}
\\
&\le \sum_{k\le 2N} \sum_{ m+n=k}\langle \phi_{n}[t]\ast_\rho f,\phi_{m}[t]\ast_\rho f\rangle_{\mathcal{H}}
= \pi_{\le 2N}[t]
\end{align*}
It follows that the abscissa of convergence of $\pi[t]$ is $\gamma$. 

Following \cite{urbanski}, choose $d_N$ slowly varying so that
$$
\pi^d[t] := \sum_{N=1}^\infty t^{-N}d_N \sum_{m+n=N} \sum_{V\in\mathcal{W}^{A,a}_m}\sum_{v\in\mathcal{W}^{A,a}_n}R_m(Vx)R_n(vx)\langle \rho(\underline{\psi}(V))^{-1}\rho(\underline{\psi}(v))f,f\rangle_{\mathcal{H}}
$$
diverges at $\gamma$.
Now write $c_n= d_{2n}$, and observe that $c$ inherits the slowly varying property from $d$. Then for $m+n=N$ we have $c_mc_n=d_{2m}d_{2n}\ge d_N$, giving
\begin{align*}
&\pi^d_{\le N}[t] = \sum_{K\le N} t^{-K}d_K \sum_{m+n=K} \sum_{V\in\mathcal{W}^{A,a}_m}\sum_{v\in\mathcal{W}^{A,a}_n}R_m(Vx)R_n(vx)\langle \rho(\underline{\psi}(V))^{-1}\rho(\underline{\psi}(v))f,f\rangle_{\mathcal{H}}
\\
&\le \sum_{K\le N} t^{-K} \sum_{m+n=K} c_mc_n\sum_{V\in\mathcal{W}^{A,a}_m}\sum_{v\in\mathcal{W}^{A,a}_n}R_m(Vx)R_n(vx)\langle \rho(\underline{\psi}(V))^{-1}\rho(\underline{\psi}(v))f,f\rangle_{\mathcal{H}}
\\
&\le \langle \phi_c[t]\ast_\rho f, \phi_c[t]\ast_\rho f\rangle_{\mathcal{H}}.
\end{align*}
The conclusion follows.
\end{proof}


\section{Twisted measures and the Proof of Theorems \ref{theorem:transitive} and \ref{theorem:LVR}}\label{section:proof}
Recall that $A,a\in\mathcal{W}^1$ and $x\in\Sigma^+$ are fixed in Section \ref{section:growthseries}. We typically denote an arbitrary element of $\mathcal{W}^1$ by $B$.

Recall that $\Upsilon$ is defined relative to the letters $a,A\in\mathcal{W}_1$.
\begin{definition}\label{def:twistedupsilon}
We will say that a measure $\nu_{\Upsilon}$ on $\Sigma^+$ is an $\Upsilon$-twisted measure if
\begin{itemize}
\item
$\nu_{\Upsilon}(\mathds{1}_{[A]})=1$ and for any $w\in\mathcal{W}^{d,a}_{n}$ we have the property
\begin{equation}\label{eq:conformalupsilon}
(\gamma(\phi\ast_\rho f))^{n}\nu_{\Upsilon}(R_n^{-1}\mathds{1}_{[wA]}) = \Upsilon(\underline{\psi}(w));
\end{equation}
\item
for each $B\in \mathcal{W}_1$, we have
$\nu_{\Upsilon}(\mathds{1}_{[B]})<\infty$, and
for any $w\in\mathcal{W}^{d,b}_{n}$ there are constants $D_{(\ref{eq:twistedbound})}(b,B,\Upsilon)$ with
\begin{equation}\label{eq:twistedbound}
|(\gamma(\phi\ast_\rho f))^{n}\nu_{\Upsilon}(R_n^{-1}\mathds{1}_{[wB]})| \le D_{(\ref{eq:twistedbound})}(B,\Upsilon);
\end{equation}
\item
for each $B\in \mathcal{W}_1$, and each $\epsilon>0$ there is a compact set $\mathcal{K}_\epsilon\subseteq \Sigma^+$ with $\nu_{\Upsilon}(E)<\epsilon$ and $\nu(E)<\epsilon$ for any measurable $E\subset \Sigma-\mathcal{K}_\epsilon \cap [B]$. 
\end{itemize}
\end{definition}

There are two key properties of the twisted measures we highlight: lower bounds relating to $\mu$ and upper bounds relating to the operator norms of linear combinations in $\rho(g)$.

Using the conformal property (\ref{eq:conformalupsilon}), it follows in the same way as for $\nu$ that for $w\in\mathcal{W}^{d,a}_{r}$
\begin{equation}\label{eq:conformal1}
\nu_{\Upsilon}([wA])
 \le C_{(\ref{localholder})}(aA) R_{r}(wAx)\gamma^{-r} \Upsilon(\underline{\psi}(w)),
\end{equation}
and for $w\in\mathcal{W}^{d,b}_{r}$
\begin{equation}\label{eq:conformal2}
\nu_{\Upsilon}([wB])
 \le D_{(\ref{eq:twistedbound})}(B,\Upsilon)C_{(\ref{localholder})}(bB) R_{r}(wBx)\gamma^{-r} .
\end{equation}

\begin{lemma}\label{DCT}
An $\Upsilon$-twisted measure $\nu_{\Upsilon}$ with $\gamma=1$ is absolutely continuous with respect to $\nu$. Moreover one has
for any bounded $H\in L^1(\mu)$ and any $B\in\mathcal{W}_1$,
$$
\int_{[B]} \frac{1}{N}\sum_{n=1}^NH\circ\sigma^n(z) d\nu_{\Upsilon}(z) \to \mu(H)\nu([B]) \; \mathrm{as}\; N\to\infty.
$$
\end{lemma}

\begin{proof}
First assume absolute continuity. Let $H\in L^1(\mu)$, bounded, and $B\in\mathcal{W}_1$. 
Note that  $\mathds{1}_{B}H\circ \sigma^n \in L^1(\nu_{\Upsilon})$ since $\nu_{\Upsilon}$ gives finite measure to $[B]$ and since $H$ is defined up to a set of $\nu$ measure zero which also has $\nu_{\Upsilon}$ measure zero.
As $\mu$ is ergodic we have
$$
\frac{1}{N}\sum_{n=0}^{N-1}H\circ \sigma^n (z) \to \mu(H) \; \mathrm{as} \; N\to\infty;
$$
pointwise convergence for $\mu$ almost every $z$.
In particular
$$
\mathds{1}_{B}(z)\frac{1}{N}\sum_{n=0}^{N-1}H\circ \sigma^n (z) \to \mu(H)\mathds{1}_{B}(z) \; \mathrm{as} \; N\to\infty;
$$
pointwise convergence for $\mu$ almost every $z$. Since $d\mu = h\,d\nu$ with $h$ taking values in $(0,\infty)$ it follows convergence takes place for $\nu$ almost every $z$. Moreover we have the domination
$$
\mathds{1}_{B}(z)\frac{1}{N}\sum_{n=0}^{N-1}H\circ \sigma^n (z) \le \|H\|_\infty \mathds{1}_{B}(z),
$$
As $\nu_{\Upsilon}$ is absolutely continuous with respect to $\nu$ we have that and pointwise convergence of Birkhoff sums $\nu$-almost surely. We check integrability of the domination
$$
\int \mathds{1}_{B}(z)\frac{1}{N}\sum_{n=0}^{N-1}H\circ \sigma^n (z) d\nu_{\Upsilon}(z) \le \|H\|_\infty \nu_{\Upsilon}([B])<\infty.
$$
In conclusion, using the Dominated Convergence Theorem,
$$
\int \mathds{1}_{B}(z)\frac{1}{N}\sum_{n=0}^{N-1}H\circ \sigma^n (z) d\nu_{\Upsilon}(z) \to  \mu(H)\int \mathds{1}_{B}(z)d\nu_{\Upsilon} \; \mathrm{as}\; N\to\infty.
$$

Now it remains to show absolute continuity. Fix $b,B\in\mathcal{W}_1$ with $bB$ admissible, and let $\epsilon>0$. We must find $\eta>0$ for which $\nu_{\Upsilon}(E\cap [bB])<\epsilon$, for any open set $E$ with $\nu(E\cap [bB])<\eta$.
(For brevity we assume always that $E\subset [bB]$.) 
We claim that we can choose $\eta=\min(\epsilon,  D_{(\ref{eq:twistedbound})}(b,\Upsilon)C^2_{(\ref{localholder})}(bB)\epsilon)$, where $C_B$ is the constant in \ref{localgibbs}. Inequality (\ref{eq:conformal2}) simplifies to
$$
\nu_{\Upsilon}([wB])
 \le D_{(\ref{eq:twistedbound})}(B,\Upsilon)C_{(\ref{localholder})}(bB) R_{r}(wBx),
$$
for any $w\in\mathcal{W}^{d,b}$, and $x\in\sigma [B]$.

Recall the hypothesis regarding the existence of a compact set $\mathcal{K}_\epsilon$ with $\nu_{\Upsilon}(E)<\epsilon$ and $\nu(E)<\epsilon$ for any measurable $E\subseteq \Sigma-\mathcal{K}_\epsilon$. Clearly if $E\subseteq \Sigma-\mathcal{K}_\epsilon$ then we are done. It is therefore sufficient to check the bound for $E \subset \mathcal{K}_\epsilon$. The following is immediate using the compactness of $ \mathcal{K}_\epsilon$. 
\begin{claim}
There is a finite set $S\subset {W}_1$ and a sequence $w(k)\in\mathcal{W}^{d,b}$, $d\in S$ with
$$
\nu_{\Upsilon}(E) = \sum_{k=1}^\infty \nu_{\Upsilon}([w(k)B]),
$$
$$
\nu(E) = \sum_{k=1}^\infty \mu([w(k)B]).
$$
\end{claim}
(First by compactness $E$ is covered by finitely many $[d]$. Then inductively define $u(B,n)$ to be the shortest word with $[du(B,n)B]\subset E-\cup_{k<n}[du(B,k)B]$. Either the sequence terminates or $|u(B,n)|\to\infty$ by compactness of $\mathcal{K}_\epsilon$. Then, up to measure zero sets, $E=\bigcup_n [du(B,n)B]$.) 

Using the claim, for a cylinder $[wB]$ we have
\begin{align*}
\nu_{\Upsilon}([wB])&\le D_{(\ref{eq:twistedbound})}(b,\Upsilon)C_{(\ref{localholder})}(bB) R_{|w|}(wBx)
\\
&\le D_{(\ref{eq:twistedbound})}(b,\Upsilon)C^2_{(\ref{localholder})}(bB) \nu([wB]).
\end{align*}
The conclusion regarding absolute continuity follows.
\end{proof}

\begin{lemma}  There is a constant $F_{(\ref{eq:qsum})}$ such that any twisted measure with $\gamma=1$ satisfies the following. Given a finite subset $\mathcal{Q}$ of $\mathcal{W}^{a,A}$, and given  $q(w)\in(0,\infty)$ for each $awA\in\mathcal{Q}$, we have, for $x\in [A]$,
\begin{equation}\label{eq:qsum}
\int \sum_{awA\in\mathcal{Q}}\frac{q(w)}{R_{|aw|}(awx)}\mathds{1}_{[awA]}(\sigma^r z) \mathds{1}_{[A]}(z) d\nu_{\Upsilon}(z) \le F_{(\ref{eq:qsum})} \left\|\sum_{awA\in\mathcal{Q}}q(w)\rho(\underline{\psi}(w))\right\|_{\mathcal{H}\to \mathcal{H}}.
\end{equation}
\end{lemma}
\begin{proof}
First note that we have a pointwise inequality for functions
$$
\frac{1}{R_{|aw|}(awx)}\mathds{1}_{[awA]}(\sigma^r z) \mathds{1}_{[A]}(z)  \le C_{(\ref{localholder})}(aA)\frac{1}{R_{|aw|}(z)}\mathds{1}_{[awA]}(\sigma^r z) \mathds{1}_{[A]}(z).
$$
Moreover, for each $v$ with $vawA$ admissible, we have, for any $y=awx\in [awA]$,
$$
\frac{1}{R_{|aw|}(z)}\mathds{1}[vawA](z) \le C_{(\ref{localholder})}(aA) R_{|v|}(vawx) \frac{1}{R_{|vaw|}(z)}\mathds{1}_{[vawA]}(z).
$$
Using these estimates and the conformality criterion we have,
\begin{align*}
&\int \sum_{awA\in\mathcal{Q}}\frac{q(w)}{R_{|aw|}(awx)}\mathds{1}_{[w]}\circ\sigma^r(z) \mathds{1}_{[B]}(z) d\nu_{\Upsilon}(z)
\\
&\le
C^2_{(\ref{localholder})}(aA)  \sum_{awA\in\mathcal{Q}}q(w)\sum_{va\in\mathcal{W}^{B,a}_{r+1}}R_{r}(vawx)\Upsilon( \rho(\underline{\psi}(vw)))
\end{align*}
Note that there is a constant $E_{(\ref{eq:vsum})}$ such that
\begin{equation}\label{eq:vsum}
\sum_{va\in\mathcal{W}^{B,a}_{r+1}}R_{r}(vawx) \le L^r 1_{[a]}(x)\le E_{(\ref{eq:vsum})} <\infty,
\end{equation}
where $L$ is the transfer operator associated to $\log R$, and using positivity $L1_{[a]}\le E^\prime h$ for some constant $E^\prime$.
It follows that
$$ 
\sum_{v\in\mathcal{W}^{B,a}_r}R_{r}(vawx)\Upsilon( \rho(\underline{\psi}(vw)))\le E_{(\ref{eq:vsum})}\Upsilon( \rho(\underline{\psi}(v_w)\rho(\underline{\psi}(w))),
$$
for some $v_w$ depending on $w$.
Since there are only finitely many possible $v_w$ we conclude that
$$ 
\sum_{v\in\mathcal{W}^{B,a}_r}R_{r}(vawx)\Upsilon( \rho(\underline{\psi}(vw)))\le E_{(\ref{eq:vsum})}\Upsilon( \rho(S)\rho(\underline{\psi}(w))),
$$
where $\rho(S)$ indicates a finite sum of some $\rho(s)$.
Now we have
\begin{align*} 
& \sum_{awA\in\mathcal{Q}}q(w)\sum_{v\in\mathcal{W}^{B,a}_r}R_{r}(vawx)\Upsilon( \rho(\underline{\psi}(vw)))
\\
&\le
E_{(\ref{eq:vsum})}\Upsilon\left( \rho(S) \sum_{awA\in\mathcal{Q}}q(w)\rho(\underline{\psi}(w))\right)
\\
&\le
E_{(\ref{eq:vsum})}\left\| \sum_{awA\in\mathcal{Q}}q(w)\rho(\underline{\psi}(w))\right\|_{\mathcal{H}\to\mathcal{H}}.
\end{align*}
\end{proof}

We are now ready to prove the theorems.

\begin{proof}[Proof of Theorem \ref{theorem:transitive}]
We let $(\rho,\mathcal{H})$ denote the left regular representation of $G$. By hypothesis there is some vector $f$ with $\gamma(\phi \ast_\rho f)=1$. We fix some $\xi\in [a]$ and note $h(\xi)>0$. Using transitivity we may choose
$\mathcal{Q}$ and $q(w)$ with $\sum_{awA\in\mathcal{Q}} q(w)\rho(\underline{\psi}(w)) = M_p^k$, where $M_p$ is a symmetric random walk operator on $G$. Moreover assume $w$ to all have $w\wedge \xi > d$ for some $d$.
We assume that $k$ and $d$ are large enough with 
$$
\nu_{\Upsilon}([A]) C^{-1} \nu([A]) (h(\xi)-2C_{(\ref{equation:lip})}(h)\theta^{d}) \le F_{(\ref{eq:qsum})}\|M_p^k\|_{\mathcal{H}\to \mathcal{H}}.
$$ 

Set $H=\sum_{awA\in\mathcal{Q}} \frac{q(w)}{R_{|w|}(awx)}\mathds{1}_{[awA]}$. Note that $H$ is bounded because $\mathcal{Q}$ is finite.
By Lemma \ref{DCT},
$$
\int \frac{1}{N}\sum_{n=0}^{N-1}H\circ \sigma^n (z)\mathds{1}_{[A]} (z) d\nu(z) \to \mu(H)\nu_{\Upsilon}([A]),
$$
as $N\to \infty$, and moreover $\nu_{\Upsilon}([A])=1$. We compute, using \ref{localgibbsmu}, that 
\begin{align*}
\mu(H) &= \sum_{awA\in\mathcal{Q}} \frac{q(w)}{R_{|aw|}(awx)} \int \mathds{1}_{[awA]}hd\nu
\\
&\ge  C_{(\ref{localholder})}(aA)^{-1}\nu([A])\sum_{awA\in\mathcal{Q}}  q(w)(h(wy)-C_{(\ref{equation:lip})}(h)\theta^{w\wedge y})  
\\
&\ge C_{(\ref{localholder})}(aA)^{-1}\nu([A])\sum_{awA\in\mathcal{Q}}  q(w)(h(\xi)-2C_{(\ref{equation:lip})}(h)\theta^{d}) .
\end{align*}
This is a contradiction to the upper bound.
\end{proof}

\begin{proof}[Proof of Theorems \ref{theorem:LVR}]
Aiming for a contradiction, suppose $\sup_{H\in\mathcal{F}}\gamma(\mathds{1}_{e_{G/H}})=1$, whence there exists $f$ with $\gamma(\phi \ast_\rho f)=1$, where $(\rho,\mathcal{H})$ is the corresponding direct sum representation. 
We proceed similarly as for the proof of Theorem \ref{theorem:transitive}. This time we use the visibility hypothesis to choose $\mathcal{Q}$ and $q(w)$ so that there is finite set $F$ with
$$
\left\|  \sum_{awA\in\mathcal{Q}} q(w)\rho(\underline{\psi}(w))\right\|_{\mathcal{H}\to \mathcal{H}}\le (\#F)^2 \|M_p^k\|_{\mathcal{H}\to \mathcal{H}},
$$
which is possible since for $v_1,v_2\in\mathcal{H}_+$ we have
$$
\left\langle \sum_{awA\in\mathcal{Q}} q(w)\rho(\underline{\psi}(w))v_1,v_2\right\rangle_{\mathcal{H}} \le \left\langle \sum_{u_1,u_2\in F} \rho(u_1)M_p^k\rho(u_2) v_1,v_2\right\rangle_{\mathcal{H}},
$$
and since these are sufficient to compute the norms of the operators.
\end{proof}

\section{Approximating measures: preliminaries}\label{section:onesidedprelim}
Throughout $f$ is fixed and $c$ and $t_q$ are given as in Section \ref{section:growthseries}. We denote $\gamma = \gamma(\phi\ast_\rho f)$ for brevity.
\subsection{Weighted Dirac mass construction}\label{subsection:diracmass}

\begin{definition}\label{def:diracmass}(Approximating measures)
Formally define, for $t>\gamma$ and $N\in\mathbb{N}$,
$$
\nu_{c;N}[t]
=\sum_{B\in\mathcal{W}_1} \sum_{n=1}^{N}  t^{-n}c_n\sum_{v\in\mathcal{W}^{B,a}_n} R_n(vx) \frac{\langle\rho(\underline{\psi}(v))f,\phi_{c;\le N}[t]\ast_\rho f\rangle_{\mathcal{H}}}{\|\phi_c[t]\ast_\rho f\|_{\mathcal{H}}^2} D(vx),
$$
in which $D(z)$ is the Dirac mass at $z\in\Sigma^+$, and $x\in\Sigma^+$ is chosen with $x\in[A]$ and $\sigma x$ not containing $A$.
\end{definition}
(The final choice relates to a remainder term in Proposition \ref{prop:main} and to the choice of first returns in the proof of Theorem \ref{theorem:tightstatement}.)

For technical reasons we will also make reference to the following family of measures, defined with respect to a group element $g\in G$ and subexponentially growing function $d:\mathbb{N}\to\mathbb{R}_+$.
Formally define, for $t>\gamma$ and $N\in\mathbb{N}$,
\begin{equation}\label{equation:variants}
\nu_{d;N}^{g^{-1}}[t]
=\sum_{B\in\mathcal{W}_1} \sum_{n=1}^{N}  t^{-n}d_n\sum_{v\in\mathcal{W}^{B,a}_n} R_n(vx) \frac{\langle\rho(\underline{\psi}(v))f,\rho(g)^{-1}\phi_{c;\le N}[t]\ast_\rho f \rangle_{\mathcal{H}}}{\|\phi_c[t]\ast_\rho f\|_{\mathcal{H}}^2} D(vx)
\end{equation}
for the same $x\in\Sigma^+$ as in Definition \ref{def:diracmass}. 

The $\nu_{c;N}[t]$ converges as a measure follows in the same way as for $\nu_N[t]$ given in equation Section \ref{section:existence}. The proof is subsumed in Theorem \ref{theorem:tightstatement}.

\begin{lemma}\label{lemma:tseries}
Let $g\in G$. For each $B$ we have that the linear functional $\nu_{c;N}^{g}[t]$ has
$$
\sup_{t>\gamma}\sup_{N\in\mathbb{N}}\nu_{c;N}^{g}[t]([B]) <\infty.
$$
If $T_{\psi}$ is transitive then
$$
\liminf_{t\to\gamma}\sup_{N\in\mathbb{N}}\nu_{c;N}^{g}[t]([B]) = \liminf_{t\to\gamma}\lim_{N\to\infty}\nu_{c;N}^{g}[t]([B]) >0.
$$
In general we have
$$
\sup_{N\in\mathbb{N}}\nu_{c;N}[t]([A]) = \lim_{N\to\infty}\nu_{c;N}[t]([A]) =1
$$
for all $t>\gamma$.
\end{lemma}
\begin{proof}
It is immediate that
$$
\nu_{c;N}[t]([B]) = \frac{\langle \phi^{B,a}_{c;\le N}[t]\ast_\rho f, \phi_{c;\le N}[t]\ast_\rho f\rangle_{\mathcal{H}}}{\|\phi_c[t]\ast_\rho f\|_{\mathcal{H}}^2}.
$$
We use Lemma \ref{lemma:strongconvergence}, the Cauchy--Schwarz inequality and Lemma \ref{lemma:changeletter} to get the upper bound 
$$
\nu_{c;N}[t]([B]) \le  \frac{\|\phi^{B,a}_{c}[t]\ast_\rho f\|_{\mathcal{H}}\|\phi_c[t]\ast_\rho f\|_{\mathcal{H}} }{\|\phi_c[t]\ast_\rho f\|_{\mathcal{H}}^2}\le D_{(\ref{equation:changeletter})}(c,A,B).
$$
In the transitive case,
\begin{align*}
&\nu_{c;N}[t]([B]) = \frac{\langle \phi^{B,a}_{c;\le N}[t]\ast_\rho f, \phi_{c;\le N}[t]\ast_\rho f\rangle_{\mathcal{H}}}{\|\phi_c[t]\ast_\rho f\|_{\mathcal{H}}^2}
\\
&\ge D_{(\ref{equation:changeletter2})}(c,A,B,g)^{-1}\frac{\langle \phi^{A,a}_{c;\le N-K_B}[t]\ast_\rho f, \phi_{c;\le N}[t]\ast_\rho f\rangle_{\mathcal{H}}}{\|\phi_c[t]\ast_\rho f\|_{\mathcal{H}}^2}
\\
&\ge D_{(\ref{equation:changeletter2})}(c,A,B,g)^{-1}\frac{\langle \phi^{A,a}_{c;\le N}[t]\ast_\rho f, \phi_{c;\le N}[t]\ast_\rho f\rangle_{\mathcal{H}}}{\|\phi_c[t]\ast_\rho f\|_{\mathcal{H}}^2}
\\
&-D_{(\ref{equation:changeletter2})}(c,A,B,g)^{-1}\frac{\langle \phi^{A,a}_{c;\le K_B}[t]\ast_\rho f,\phi_{c;\le N}\ast_\rho f\rangle_{\mathcal{H}}}{\|\phi_c[t]\ast_\rho f\|_{\mathcal{H}}}.
\end{align*}

We check that the linear functionals $\nu_{c;N}^{g^{-1}}[t]$ converge as $N\to \infty$. (So we need to check that they are Cauchy.)
For arbitrary $N,M$ we have
\begin{align*}
&\nu_{c;N}^{g^{-1}}[t]([B]) - \nu_{c;M}^{g^{-1}}[t]([B])
\\
&= 
\frac{\langle \phi^{B,a}_{c;\le N}[t]\ast_\rho f- \phi^{B,a}_{c;\le M}[t]\ast_\rho f,\rho(g)^{-1} \phi_c[t]\ast_\rho f\rangle_{\mathcal{H}}}{\|\phi_c[t]\ast_\rho f\|^2}
\\
&\le 
\frac{\|\phi^{B,a}_{c;\le N}[t]\ast_\rho f- \phi^{B,a}_{c;\le M}[t]\ast_\rho f\|}{\|\phi_c[t]\ast_\rho f\|^2}
.
\end{align*}
Strong convergence (Lemma \ref{lemma:strongconvergence}) gives that $\|\phi^{B,a}_{c;\le N}[t]\ast_\rho f- \phi^{B,a}_{c;\le M}[t]\ast_\rho f\|\le \epsilon$ for $N,M$ sufficiently large.
\end{proof}

\subsection{Basic estimates for the approximating measures}\label{subsection:basic}
\begin{proposition}[The main equality]\label{prop:main}
Let $wB\in\mathcal{W}$ with $|w|=k$. For any $r\ge k$ we have that $\nu_{c;N}^{g^{-1}}[t](t^{k}R^{-1}_{k}\mathds{1}_{[wB]})$ equals
\begin{align*}
 &\sum_{n=1}^{N-k}t^{-n}c_n\frac{c_{n+k}}{c_n}
\sum_{v\in\mathcal{W}_n^{B,a}}
  \frac{\langle \rho(\underline{\psi}(w))\rho(\underline{\psi}(v))f,\rho(g)^{-1}\phi_{c;\le N}[t]\ast_\rho f\rangle_{\mathcal{H}}}{\| \phi_{c} [t]\ast_\rho f\|_{\mathcal{H}}^2} R_{n}(vx)
\\
&+ 
 \sum_{n=1}^k t^{-n}c_n\sum_{b\in\mathcal{W}_1}\sum_{v\in\mathcal{W}_n^{b,a}}R_{n}(vx)
 \frac{\langle \rho(\underline{\psi}(v))f,\phi_{c;\le N}[t]\ast_\rho f\rangle_{\mathcal{H}}}{\| \phi_{c} [t]\ast_\rho f\|_{\mathcal{H}}^2 } \frac{t^{k}\mathds{1}_{[wB]}(vx)}{R_{k}(vx)}
\end{align*}
\end{proposition}

\begin{proof}
Write $wB=duB$ where $d\in\mathcal{W}_1$. We have that $\nu_{c;N}^{g^{-1}}[t](t^{k}R_{k}^{-1}\mathds{1}_{[duB]}) $ equals
 \begin{align*}
& \sum_{n=k+1}^Nt^{-n}c_n\sum_{v_1\cdots v_k v\in\mathcal{W}_n^{d,a}}R_{n}(v_1\cdots v_k vx)
 \frac{\langle \rho(\underline{\psi}(v_1\cdots v_k v))f,\phi_{c;\le N}[t]\ast_\rho f\rangle_{\mathcal{H}}}{\| \phi_{c} [t]\ast_\rho f\|_{\mathcal{H}}^2 } \frac{t^{k}\mathds{1}_{[duB]}(v_1\cdots v_k vx)}{R_{k}(v_1\cdots v_kvx)}
 \\
 &+
 \sum_{n=1}^k t^{-n}c_n\sum_{v\in\mathcal{W}_n^{d,a}}R_{n}(vx)
 \frac{\langle \rho(\underline{\psi}(v))f,\phi_{c;\le N}[t]\ast_\rho f\rangle_{\mathcal{H}}}{\| \phi_{c} [t]\ast_\rho f\|_{\mathcal{H}}^2} \frac{t^{k}\mathds{1}_{[duB]}(vx)}{R_{k}(vx)}
 \end{align*}
First sum simplifies since 
$$
R_{n}(v_1\cdots v_k vx)/R_{k}(v_1\cdots v_kvx)=R_{n-k}(vx)
$$ 
and $\mathds{1}_{[duB]}(v_1\cdots v_k vx)\ne 0$ forces $v\in\mathcal{W}_n^{B,a}$.
\end{proof}

In order to even check that the sequence of approximating measures is tight we will need to check some approximation of the local Gibbs inequality.

\begin{lemma}[$c$-heavy RHS local Gibbs]\label{one}
For each $B$, $g\in G$, and $w\in\mathcal{W}^{B,a}_r$ we have
\begin{align*}
&\nu_{c;N}^{g^{-1}}[t]([wA])
 \le R_{r}(wx)t^{-r} C_{(\ref{localholder})}(aA)C_{(\ref{equation:changeletter})}(c,A,B) \sup_{n\in\mathbb{N}}\left( \frac{c_{n+r}}{c_n}\right).
\end{align*}
\end{lemma}
\begin{proof}
We focus on bounding $\nu_{c;N}^{g^{-1}}[t](\gamma^rR_r^{-1}\mathds{1}_{[wA]})$ since we know how to use (\ref{localholder}) to finish the job.

Recall that Proposition \ref{prop:main} gives an expression for the term $\nu_{c;N}^{g^{-1}}[t](\gamma^rR_r^{-1}\mathds{1}_{[wA]})$. As we chose $\sigma x$ to not contain $A$ this forces that the second series is $0$ except at $n=r$ from which
\begin{align*}
&\sum_{n=1}^r t^{-n}c_n\sum_{b\in\mathcal{W}_1}\sum_{v\in\mathcal{W}_n^{b,a}}R_{n}(vx)
 \frac{\langle \rho(\underline{\psi}(v))f,\phi_{c;\le N}[t]\ast_\rho f\rangle_{\mathcal{H}}}{\|\phi_{c} [t]\ast_\rho f\|_{\mathcal{H}}^2} \frac{t^{r}\mathds{1}_{[wA]}(vx)}{R_{r}(vx)}
\\
&=
 \frac{\langle \rho(\underline{\psi}(w))f,\phi_{c;\le N}[t]\ast_\rho f\rangle_{\mathcal{H}}}{\|\phi_{c} [t]\ast_\rho f\|_{\mathcal{H}}^2}
 \le 1.
\end{align*}

We give an upper bound for the first series in Proposition \ref{prop:main} by
\begin{align*}
&
\left(\sup_{n\ge r} \frac{c_{n+r}}{c_n}\right) \sum_{n=1}^Nt^{-n}c_n \sum_{v\in\mathcal{W}_n^{A,a}}
 \frac{\langle \rho(\underline{\psi}(w))\rho(\underline{\psi}(v))f,\rho(g)^{-1}\phi_{c;\le N}[t]\ast_\rho f\rangle_{\mathcal{H}}}{\|\phi_{c} [t]\ast_\rho f\|_{\mathcal{H}}^2} R_{n}(vx)
 \\
&= \left(\sup_{n\ge k} \frac{c_{n+r}}{c_n}\right)
 \frac{\langle  \rho(\underline{\psi}(w)) \phi^{B,a}_{c;\le N}[t]\ast_\rho f,\phi_{c;\le N}[t]\ast_\rho f\rangle_{\mathcal{H}}}{\|\phi_{c} [t]\ast_\rho f\|_{\mathcal{H}}^2} R_{n}(vx)
 \\
&\le
\left(\sup_{n\ge k} \frac{c_{n+r}}{c_n}\right)
\frac{\|\phi^{B,a}_c[t]\ast_\rho f\|_{\mathcal{H}}}{\|\phi_{c} [t]\ast_\rho f\|_{\mathcal{H}}}
\le \left(\sup_{n\ge k} \frac{c_{n+r}}{c_n}\right)C_{(\ref{equation:changeletter})}(c,A,B).
\end{align*}
\end{proof}

\section{Approximating measures: tightness}\label{section:existence}
\begin{theorem}\label{theorem:tightstatement}
Suppose $\gamma>\gamma(\mathrm{SPR})$.
For every $B$ and for every $\epsilon>0$ there is a compact set $\mathcal{K}$ for which
$$
\nu_{c;N}^{g^{-1}}[t]((\Sigma^+ - \mathcal{K})\cap [B])\le  \epsilon
$$
and
$$
\nu_N[t^\prime]((\Sigma^+ -\mathcal{K})\cap [B])\le  \epsilon
$$
for every $N\in\mathbb{N}$, $g\in G$ and $t>\gamma$, $t^\prime >1$.
\end{theorem}

We proceed in a similar fashion as Sarig \cite{Sarig}. In our case we rely on Lemma \ref{one} to estimate the measure of cylinders.

Without loss of generality we may assume that $\Sigma^+ \subset \mathbb{N}^\mathbb{N}$; that is we represent the transitions by natural numbers. Let us assume that $A$ is represented by $1\in\mathbb{N}$.
We need notation governing the first returns. Let $B\in\mathcal{W}_1$. For a sequence of numbers $(M_p)_{p\in\mathbb{N}}$ set:
\begin{itemize}
\item
$\mathcal{K}((M_p)) = \left\{ x\in\Sigma : \forall p\in\mathbb{N} \, x_p\le M_p \right\}.$
\item $\mathcal{B}_n = \left\{ w : w=w_1\cdots w_{n-1}a\in\mathcal{W}_n, \tau(B,w_1)=1, w_i>1, i=1,\cdots ,n-1\right\}$
\item $\mathcal{A}_n = \left\{ w : w=w_1\cdots w_{n-1}a\in\mathcal{W}_n, \tau(A,w_1)=1, w_i>1 , i=1,\cdots ,n-1\right\}$;
\item $\mathcal{A} = \bigcup_{n=1}^\infty \mathcal{A}_n$;
\item for $k\in\mathbb{N}$, 
$\mathcal{A}_n(\ge M_k)= \left\{ w : w=w_1\cdots w_n\in\mathcal{A}_n  \exists i=1,\cdots ,n \,:\,w_i> M_k \right\}$;
\item
$\mathcal{A}(\ge M_k)= \left\{ w : w=w_1\cdots w_n\in\mathcal{A}_n, n\in \mathbb{N},  \exists i=1,\cdots ,n\,:\, w_i> M_k \right\}$;
\item
$\mathcal{B}(\ge M_k)= \left\{ w : w=w_1\cdots w_n\in\mathcal{B}_n, n\in \mathbb{N},  \exists i=1,\cdots ,n\,:\, w_i> M_k \right\}$;
\item 
$\mathcal{A}^j$ is the usual cartesian product.
\end{itemize}

We write $\hat{c}:\mathbb{N}\to\mathbb{R}$
$$
\hat{c}_r = \sup_{n\in\mathbb{N}}\left(\frac{c_{n+r}}{c_r} \right)
$$
We write 
$$
\eta^{A,a}[t]:=\sum_{n=1}^\infty t^{-n} \hat{c}_n \sum_{w\in\mathcal{A}_n}R_n(Aw_1\cdots w_{n-1}x),
$$
and
$$
\xi^{B,a}[t]:=\sum_{n=1}^\infty t^{-n} \hat{c}_n\sum_{w\in\mathcal{B}_{n}}R_{n}(Bw_1\cdots w_{n-1}x),
$$
recalling that $x\in [A]$ (see Definition \ref{def:diracmass}).

The set $\mathcal{K}((M_p))$ is compact. And 
$$
\Sigma^+ - \mathcal{K}((M_p)) = \left\{ x\in\Sigma : \exists q\in\mathbb{N}  \, x_q> M_q \right\}.
$$
We will always assume that $M_{k}\le  M_{k+1}$.

For brevity we write $m_N[t]=\nu_{c;N}^{g^{-1}}[t]$. The only tool we use is Lemma \ref{one} which applies uniformly in $g$.
\begin{claim}
Assuming $M_p$ are large enough, if $z\in \mathrm{supp}(m_N[t])$ and $z\in (\Sigma - \mathcal{K}((M_p)))\cap [B]$, it must be that either 
\begin{itemize}
\item[1]
$z\in[BuA]$, with $u\in\mathcal{B}(\ge M_1)$, or
\item[j]
$z\in [BuAvwA]$ with $u\in\mathcal{B}$, $v\in(\mathcal{A} A) ^j$, and $w\in \mathcal{A}(\ge M_j)$, for some $j\in\mathbb{N}$.
\end{itemize}
\end{claim}

Recall the constant $E(aA)\hat{c}_r=C_{(\ref{localholder})}(aA)C_{(\ref{equation:changeletter})}(c,A,B)(B)\hat{c}_r$ appearing in Lemma \ref{one}. (We suppress the dependency on $B$ which is fixed but include the dependency on $aA$ for clarity.) We use that $\hat{c}_{m+k}\le \hat{c}_{k}\hat{c}_{m}$
\begin{claim}\label{meas} We have
\begin{itemize}
\item[1:]
For $u\in\mathcal{A}(\ge M_1)$, $|u|=r-1$,
$$
m_N[t]([BuA]) \\
\le E(aA) t^{-r}\hat{c}_r R_{r}(ux)
$$
\item[j:]
For $u\in\mathcal{B}(\ge M_1)$, $|u|=r-1$,  $v\in(\mathcal{A}A)^j$, $|v|=p$, $w\in \mathcal{A}_k(\ge M_j)$, for some $j\in\mathbb{N}$;
\begin{align*}
&m_N[t]([BuAvwA]) \le E(aA) t^{-(r+1+p+k)}\hat{c}_{r+1+p+k}R_{r+1+p+k}(BuAvwx)
\\
&\le E(aA)\hat{c}_r\hat{c}_p\hat{c}_k  t^{-r} t^{-p}t^{-k} C_{(\ref{localholder})}(aA) R_{r}(Bux) C_{(\ref{localholder})}(aA) R_{p}(Avx)
C_{(\ref{localholder})}(aA)  R_{k}(Awx).
\end{align*}
\end{itemize}
\end{claim}
(Noting $x\in[A]$ in the final local H\"{o}lder estimates.)

\begin{claim}\label{union} There is a constant $Q$ such that
\begin{itemize}
\item[1:]
\begin{align*}
&m_N[t](\left\{ [BuA] : u\in\mathcal{B}_k(\ge M_1), k\in\mathbb{N}\right\} ) \\
&\le 
Q
\sum_{k\in\mathbb{N}} t^{k-1}\hat{c}_{k-1} \sum_{u\in\mathcal{B}_k(\ge M_1)}R_{k-1}(Bux) 
\end{align*}
\item[j:]
\begin{align*}
&m_N[t](\left\{ [uvw] : v\in\mathcal{A}^j, u\in\mathcal{B}, w\in\mathcal{A}_k(\ge M_j), j,k\in\mathbb{N}\right\} ) \\
&\le Q \xi^B(t)
\left(Q
\eta^A[t]\right)^{2j+1}
\sum_{k\in\mathbb{N}}t^k\hat{c}_{k} \sum_{w\in\mathcal{A}_k(\ge M_j)}R_{k}(Awx) 
\end{align*}
\end{itemize}
\end{claim}

\begin{claim}
For every $\lambda>0$ there is choice of $K_\lambda$ (uniform in $t>\gamma$) so that
$$
\sum_{k\in\mathbb{N}}t^k\hat{c}_{k} \sum_{w\in\mathcal{A}_k(\ge K_\lambda)}R_{k}(Awx) \le \lambda
$$
and
$$
\sum_{k\in\mathbb{N}}t^k\hat{c}_{k} \sum_{w\in\mathcal{B}_{k}(\ge K_\lambda)}R_{k}(Bwx) \le \lambda.
$$
\end{claim}

\begin{proof}[Proof of Theorem \ref{theorem:tightstatement}]
 Let $\epsilon>0$. For $q\in\mathbb{N}$, set 
$$
\lambda_q = \frac{\epsilon}{2^{q}}(Q)^{-1}
\left(Q
\eta^A[ t]\right)^{-2q-1}
$$
and let $K_{\lambda_q}$ be given as in the claim. Set $M_q = K_{\lambda_q}$. Then
\begin{align*}
&m_N[t](\Sigma - \mathcal{K}((M_p)))\\
&\le \sum_{q\in\mathbb{N}}
Q
\left(Q)
\eta^A[t]\right)^{2q+1}
\sum_{k\in\mathbb{N}}t^k\hat{c}_k \sum_{w\in\mathcal{A}_k(\ge M_q)}R_{k}(wx)
\\
&\le \epsilon \sum_{q\in\mathbb{N}}\frac{1}{2^{q}} \le \epsilon
\end{align*}
\end{proof}


One does have to be careful about extracting accumulation points when $\Sigma^+$ is non-compact --- it is the tightness result of Theorem \ref{theorem:tightstatement} that implies the existence of accumulation points. Let us sketch of the details of  this well-known mechanism. Suppose $m_k$ are tight in $k\in\mathbb{N}$. 
For a fixed $B$ and compact set $\mathcal{K}_n$ in the tightness criterion, we have that the measures $m_k$ restricted to $\mathcal{K}_n\cap [B]$ have an accumulation point that is a positive measure. So, for some subsequence $m_{k_n(q)}$ converge as $q\to\infty$ to a positive measure. Using nesting of $\mathcal{K}_n$ we may assume that $k_n(q)$ is a subsequence of $k_{n-1}(q)$. Then setting $q_n = k_n(n)$ we have that $m_{q_n}$ converges for any $\mathcal{K}_n$, and in particular the limit measure is well-defined on the union, which is $[B]$.
We use a similar diagonal sequence have convergence along a sequence which works for all $B$.
We use a similar diagonal sequence to deduce convergence along a sequence $t_k\to \gamma$ which works for all $g$. To see that the limit is finite on cylinders we using the first part of Lemma \ref{lemma:tseries} which tells us that the measure of a cylinder is bounded from above in $t$.
\begin{corollary}\label{cor:tight}
Assume that $\gamma>\gamma(\mathrm{SPR})$. For each $t>\gamma$ there are measures $\nu^{g^{-1} }_{c}[t]$ on $\Sigma^+$, finite on cylinders, with $\nu^{g^{-1} }_{c;N}[t](E)\to\nu^{g^{-1} }_{c}[t](E)$ on any cylinder $E$ as $N\to \infty$ in the weak-$\ast$ topoology. 
in this way
\[
\nu^{g^{-1}}_{c}[t] =\sum_{b\in\mathcal{W}} \sum_{n=1}^\infty  t^{-n}d_n\sum_{v\in\mathcal{W}^{b,a}_n} R_n(vx)\frac{\langle\rho(\underline{\psi}(v))f,\rho(g)^{-1}\phi_c[t]\ast_\rho f \rangle_{\mathcal{H}}}{\|\phi[t]_c\ast_\rho f\|_{\mathcal{H}}^2} D(vx)
\]
is well defined, and for each $B$ the $\nu_{c;N}[t]$ measure of $[B]$ is bounded uniformly in $t>\gamma$.

There is a sequence $t_k\to \gamma$ and measures $\nu^{g^{-1} }_{c}$ on $\Sigma^+$, finite on cylinders, with $\nu^{g^{-1} }_{c}[t_k](E)\to \nu^{g^{-1} }_{c}(E)$ on any cylinder $E$ as $k\to\infty$ in the weak-$\ast$ topology. In addition $\nu_{c}([A])=1$.
\end{corollary}

\section{Limiting measure as a twisted measure}\label{section:twistproof}
\subsection{Properties of limit measure}
In order to elaborate on properties of $\nu_c$ and $\nu_c[t]$ as given in Corollary \ref{cor:tight} it is useful to consider other slowly varying functions.
\begin{lemma}\label{lemma:dtight}
Assume $\gamma>\gamma(\mathrm{SPR})$. Let $d:\mathbb{N}\to [0,\infty)$ with 
$$
\max(\sup \left\{ \frac{d_k}{c_k} : c_k\ne 0 \right\}, \sup \left\{ \frac{c_k}{d_k} : d_k\ne 0 \right\})<\infty.
$$
Then $\nu_{d;N}^{g^{-1}}[t]$ converge to a measure, finite on cylinders, as $N\to\infty$. In this way
\[
\nu^{g^{-1}}_{d}[t] =\sum_{b\in\mathcal{W}} \sum_{n=1}^\infty  t^{-n}d_n\sum_{v\in\mathcal{W}^{b,a}_n} R_n(vx)\frac{\langle\rho(\underline{\psi}(v))f,\rho(g)^{-1}\phi_c[t]\ast_\rho f \rangle_{\mathcal{H}}}{\|\phi[t]_c\ast_\rho f\|_{\mathcal{H}}^2} D(vx)
\]
is well defined, and for each $B$ the $\nu_{d;N}[t]$ measure of $[B]$ is bounded uniformly in $t>\gamma$.
\end{lemma}
\begin{proof}
That $\nu^{g^{-1}}_{d}[t]$ is well-defined follows from the tightness result of the previous section.
Letting $C = \max(\sup \left\{ \frac{d_k}{c_k} : c_k\ne 0 \right\}, \sup \left\{ \frac{c_k}{d_k} : d_k\ne 0 \right\})$, we can transfer any cylinder bounds for $\nu_{c;N}[t]$ to $\nu_{d;N}[t]$ since $C^{-1} \nu_{d;N}[t](E) \le \nu_{c;N}[t](E) \le C \nu_{d;N}[t](E)$ for every open set.
\end{proof}

``The main equality" of Proposition \ref{prop:main} can be interpreted as saying that 
$$
\int Ft^kR^{-1}_k d\nu_{c;N}[t] = \nu_{d;N}^{g^{-1}}[t](F) + \mathrm{rem.}^{g^{-1}}_{d;N;[1,k]}[t](t^kR^{-1}_kF) +\mathrm{rem.}^{g^{-1}}_{d;N;[N-k, N]}[t](F)
$$
with $d_n = \mathds{1}_{[k,\infty)}(n)c_{n+k}$ and $\mathrm{rem.}^{g^{-1}}_{d;N;I}[t]$ the measure
$$
\mathrm{rem.}^{g^{-1}}_{d;N;I}[t] = \sum_{n\in I} d_n t^{-n}\sum_{v\in\mathcal{W}_{n}^{B,a}}\frac{\langle \rho(\underline{\psi}(v))f,\rho(g)^{-1}\phi_{c;\le N}[t]\ast_\rho f\rangle_{\mathcal{H}}}{\|\phi_c[t]\ast_\rho f\|_{\mathcal{H}}^2}.
$$

\begin{proposition}\label{prop:changerange}
For any $k$
$$
 \nu_{d;N}^{g^{-1}}[t] =  \nu_{\underline{d};N}^{g^{-1}}[t] +\mathrm{rem.}^{g^{-1}}_{d,N;[1,k]}[t]
$$
with $\underline{d}_n = \mathds{1}_{[k,\infty)}(n)c_{n}$
\end{proposition}

We first check that the two remainder terms to go zero.
\begin{lemma}\label{lemma:remone}
For each $g$, $r$, and $d$, we have
$$
\mathrm{rem.}^{g^{-1}}_{d,N;[1,r]}[t] \to 0 \;\mathrm{as}\; t\to \gamma.
$$
\end{lemma}

\begin{proof}
We have
$$
\mathrm{rem.}^{g^{-1}}_{d,N;[1,r]}[t] ([B]) = \frac{\langle \phi^{B,a}_{d;\le r}\ast_\rho f, \rho(g)^{-1}\phi_{c;\le N}[t]\ast_\rho f\rangle_{\mathcal{H}}}{\|\phi_c[t]\ast_\rho f\|_{\mathcal{H}}^2}.
$$
First note that $d/c$ is bounded by some $C$ in the range $[1,r]$ so
$$
\mathrm{rem.}^{g^{-1}}_{d,N;[1,r]}[t] ([B]) \le C \frac{\langle \phi^{B,a}_{c;\le r}[t]\ast_\rho f, \rho(g)^{-1}\phi_{c;\le N}[t]\ast_\rho f\rangle_{\mathcal{H}}}{\|\phi_c[t]\ast_\rho f\|_{\mathcal{H}}^2},
$$
for some $C>0$. We conclude using the Cauchy-Schwartz Lemma, Lemma \ref{lemma:strongconvergence} and Lemma \ref{lemma:changeletter}.
\end{proof}

\begin{lemma}\label{lemma:remtwo}
For each $g$, $r$ and $d$ with $d/c$ bounded, we have
$$
\mathrm{rem.}^{g^{-1}}_{d,N;[N-r,N]}[t] \to 0 \;\mathrm{as} \; N\to\infty.
$$
\end{lemma}

\begin{proof}
We have
$$
\mathrm{rem.}^{g^{-1}}_{d,N;[N-r,N]}[t] ([B]) = \frac{\langle \phi^{B,a}_{d;N-r\le n\le N}[t]\ast_\rho f, \rho(g)^{-1}\phi_{c;\le N}[t]\ast_\rho f\rangle_{\mathcal{H}}}{\|\phi_c[t]\ast_\rho f\|_{\mathcal{H}}}
$$
We have the upper bound
\begin{align*}
&\mathrm{rem.}^{g^{-1}}_{d,N;[N-r,N]}[t] ([B]) 
\\
&\le \left( \sup_{n\in[N-r,N]} \frac{d_n}{c_n}\right)C_{(\ref{equation:changeletter})}(c,A,B)\frac{\langle \rho(h_B)\phi^{A,a}_{c;N-r \le n\le N+K_{B}}[t]\ast_\rho f, \rho(g)^{-1}\phi_{c;\le N}[t]\ast_\rho f\rangle_{\mathcal{H}}}{\|\phi_c[t]\ast_\rho f\|^2_{\mathcal{H}}}.
\end{align*}
It suffices to control
$$
\frac{\|\phi^{A,a}_{c;N-r \le N+K_{B}}\ast_\rho f\| \|\phi_{c;\le N}[t]\ast_\rho f\|}{\|\phi_c[t]\ast_\rho f\|^2_{\mathcal{H}}}.
$$
Indeed, for each $t$ we have that  $\|\phi^{A,a}_{c;N-r \le N+K_{B}}[t]\ast_\rho f\|\to 0$ as $N\to \infty$ by strong convergence of $\phi_{c;\le N-r}[t]\ast_\rho f$ to $\phi_c[t]\ast_\rho f$.
\end{proof}

\begin{lemma}\label{lemma:cchange}
If $\frac{c_n}{d_n}\to 1$ as $n\to\infty$ then $\nu_{d}^{g^{-1}}[t_k]$ converge as $k\to\infty$, and the limit measure has $\nu_{d}^{g^{-1}}=\nu_{c}^{g^{-1}}$.
\end{lemma}

\begin{proof}
Let the cylinder $[w]$ be arbitrary. Let $\epsilon>0$ be arbitrary. Using the hypothesis on $d$, choose $K$ sufficiently large with $c_n/d_n\le 1+\epsilon$ for $n\ge K$. Using Proposition \ref{prop:changerange} we have
$$
 \nu_{d;N}^{g^{-1}} [t_q]([w]) =  \nu_{\underline{d};N}^{g^{-1}}[t_q]([w]) +\mathrm{rem.}^{g^{-1}}_{d,N;[1,k]}[t].
$$
Now, using the weak-$\ast$ convergence of the measures we have,
$$
 \nu_{\underline{d}}^{g^{-1}}[t_q]([w]) \le (1+\epsilon) \nu_{c}^{g^{-1}}[t_q]([w])
$$
We use Lemma \ref{lemma:remone} to conclude that the remainder term vanishes as $q\to\infty$. Since $\epsilon$ was arbitrary the conclusion follows.
\end{proof}

\subsection{Twistedness}

\begin{lemma}\label{lemma:matrixcoeff}
For any $wA$ we have
$$
\int_{[wA]}\gamma^n R_n^{-1} d\nu_{c} = \Upsilon(\underline{\psi}(w)).
$$
In general
$$
\left|\int_{[wB]} \gamma^nR_n^{-1} d\nu_{c}\right|\le C_{(\ref{equation:changeletter})}(c,A,B).
$$
and supposing $T_{\psi}$ is transitive,
$$
\int_{[wB]} \gamma^nR_n^{-1} d\nu_{c} \le C_{(\ref{equation:changeletter2})}(c,A,B,e)\Upsilon(\underline{\psi}(w)).
$$
\end{lemma}

\begin{proof}[Proof of Lemma \ref{lemma:matrixcoeff}]
By the main equality and Lemmas \ref{lemma:remone} and \ref{lemma:remtwo},
$$
\nu_{c}(\gamma^n R_n^{-1}\mathds{1}_{[wB]}) = \nu_{c}^{g^{-1}}([B]),
$$
for $g=\underline{\psi}(w)$. We are always using that fact that $R_n^{-1}\mathds{1}_{[wB]}$ is bounded and continuous so that we can deduce its integral with respect to the weak limit $\nu_{c}$.
We know that
$$
 \nu_{c}^{g^{-1}}([B])  = \lim_{q\to\infty}  \lim_{N\to\infty}\frac{\langle \rho(g)\phi^{B,a}_{c;\le N}[t_q]\ast_\rho f , \phi_{c;\le N}[t_q]\ast_\rho f\rangle_{\mathcal{H}}}{\|\phi_{c}[t_q]\ast_\rho f\|^2_{\mathcal{H}}}.
$$
The conclusions follow using strong convergence and Lemma \ref{lemma:changeletter}.
\end{proof}

\section{Further context for twisted measures}\label{section:musings}
\subsection{Measures twisted by cocycles}
The map $\sigma:\Sigma^+\to\Sigma^+$ is not invertible but on any cylinder $[w]$ with $|w|=n$ one has a local (left) inverse $\tau^{(w)}:\sigma^n [w]\to  [w]$.
\begin{definition}\label{def:twistedbycocycle}
Let $m$ be a probability measure on $\Sigma^+$. If there exist $\gamma>0$ and $h:G \times \Sigma^+\to\mathbb{R}$ satisfying
\begin{equation}\label{harmonic}
\frac{d{(\tau^{(w)})}_*m}{dm}(z) = \gamma^{n}R_n(z)^{-1}h(\underline{\psi}(w),z),
\end{equation}
for every $w\in\mathcal{W}_n$ and $z\in[w]$ then we say that $m$ is \emph{twisted} by $h$.
\end{definition}
Often one talks of unitary twists, of say, operators. We draw on the usage of measures twisted by cocycles. Though $h$ is not a cocycle (there is no group action here) we do have the identity
$$
h(e,z) = h(g^{-1},z) h(\underline{\psi}(w),\tau^{(w)}z)
$$
whenever $g  = \underline{\psi}(w)$ and $z\in\sigma^n[w]$.
This is of interest because
$$
\int_{[A]} \frac{h(g^{-1},z)}{h(e, z )} dm(z)
=\int \mathds{1}_{[wA}(z)  \gamma^n R_n^{-1}(z)dm(z).
$$
In the case that the group extension is transitive we may observe that $\nu_{\Upsilon}$ satisfies this notion of twisted and with the integral of $\frac{h(g^{-1},z)}{h(e, z )} $ equalling $\Upsilon(g)$.

\subsection{The concept of decay exponents}
Recall that $A,a\in\mathcal{W}_1$ are fixed, $c$ is fixed, and that $\phi_c[t], \Upsilon$ are defined relative to these.
Throughout we assume that $\Sigma^+$ is compact. We will always assume that the associated group extensions are transitive, and thus may upgrade \ref{lemma:changeletter} to 
\begin{equation}\label{eq:trans}
\phi_c^{B,b}[t]\le D_{(\ref{eq:trans})} \phi_c[t]\le D^2_{(\ref{eq:trans})} \phi_c^{B,b}[t],
\end{equation} 
for some constant $D_{(\ref{eq:trans})}$ independent of $B,b\in\mathcal{W}_1$. It follows that 
$$
\Upsilon^{B,b;D,d}\le D^2_{(\ref{eq:trans})} \Upsilon \le D^4_{(\ref{eq:trans})} \Upsilon^{B,b;D,d}.
$$

We begin with a reason for the terminology of decay exponent.
\begin{lemma}\label{lemma:decaystatement}
Assume $\Sigma^+$ is compact. Let $f\in\mathcal{H}_+$. If $\gamma(\phi\ast_\rho f)<1$ then for every $\gamma\in(\gamma(\phi\ast_\rho f),1)$ and every $v\in\mathcal{H}_+$ we have that for $\mu$ almost every $x\in\Sigma^+$ there is $N$ with
$$
n\ge N \; \implies \; \langle \rho(\underline{\psi}(x_{0}\cdots x_{n-1}))f,v\rangle_{\mathcal{H}} \le \gamma^n
$$
\end{lemma}

\begin{proof}[Proof of Lemma \ref{lemma:decaystatement}]
Let $f,v\in\mathcal{H}_+$ be arbitrary.
We denote the sets 
$$
E_n= \left\{ x\in\Sigma  : \langle \psi(x_{0})\cdots \psi(x_{n-1})f,v\rangle_{\mathcal{H}} \ge \gamma^n \right\},
$$
$$
\limsup E_n =  \left\{ x\in\Sigma  : \langle \psi(x_{0})\cdots \psi(x_{n-1})f,v\rangle_{\mathcal{H}} \ge \gamma^n,\, n \;\mathrm{i.o.} \right\},
$$
with ``i.o." shorthand for ``infinitely often".
Our goal is to use the Borel-Cantelli Lemma to deduce that $\limsup E_n$ has zero $\mu$ measure. 
The Gibbs property \ref{equation:gibbs} states that there is a constant $C>0$ with
$$
\mu([w])\le CR_n(wy)
$$
for any $w$ of length $n$ and $y\in\sigma^n [w]$.
 We compute 
\begin{align*}
\mu(E_n) &= \sum_{B,b\in\mathcal{W}_1}\sum_{\substack{w\in\mathcal{W}_n^{A,a}:\\ \langle \psi(w_{1})\cdots \psi(w_n)v,f\rangle_{\mathcal{H}} \ge \gamma^n}} \mu([w_1\cdots w_n])
\\
&\le  \sum_{B,b\in\mathcal{W}_1}\sum_{\substack{w\in\mathcal{W}_n^{A,a} :\\ \langle \psi(w_1)\cdots \psi(w_n)f,v\rangle_{\mathcal{H}} \ge \gamma^n}} \mu([w_1\cdots w_n.])\gamma^{-n}\langle \rho(\psi(w_1)\cdots \psi(w_n))f,v\rangle_{\mathcal{H}}
\\
&\le C \sum_{B,b\in\mathcal{W}_1}\sum_{w\mathcal{W}_n^{B,b}} R_n(w_1\cdots w_n x)\gamma^{-n}\langle \rho(\psi(w_1)\cdots \psi(w_n))f,v\rangle_{\mathcal{H}}
\\
&\le  C \sum_{B,b\in\mathcal{W}_1} \langle \phi^{B,b}_n[\gamma]\ast_\rho f,v\rangle_{\mathcal{H}},
\end{align*}
with $\phi^{B,b}_n[t](g) := \sum_{w\mathcal{W}_n^{B,b}:\underline{\psi}(w)=g} t^{-n}R_n(w_1\cdots w_n x)$.
Then
$$
\sum_{n=1}^\infty \mu(E_n) \le  \sum_{B,b\in\mathcal{W}_1} \langle \phi^{B,b}[\gamma]\ast_\rho f,v\rangle_{\mathcal{H}} \le \sum_{B,b\in\mathcal{W}_1}\|\phi^{B,b}[\gamma ]\ast_\rho f \|_{\mathcal{H}}\|v\|_{\mathcal{H}}.
$$
By Lemma \ref{lemma:strongconvergence} we know that $\|\phi^{B,b}[\gamma ]\ast_\rho f \|_{\mathcal{H}}$ is finite for $\gamma >\gamma(\phi\ast_\rho f)$, thus giving the convergence case of the Borel-Cantelli Lemma as desired.
\end{proof}

We give an interpretation that the $\nu_\Upsilon$ measure finds points that see the slowest decay of matrix coefficients. We begin with a qualitative result.
\begin{lemma}\label{lemma:sphericaldecay}
Assume $\Sigma^+$ is compact and that $T_{\psi}$ is transitive. Let $f\in\mathcal{H}_+$. If $\gamma(\phi\ast_\rho f)<1$ then for every $\epsilon>0$ and every $v\in\mathcal{H}_+$ we have that for $\nu_{\Upsilon}$ almost every $x\in\Sigma^+$ there is $N$ with
$$
n\ge N \; \implies \; \langle \rho(\underline{\psi}(x_{1}\cdots x_{n}))f,v\rangle_{\mathcal{H}} \le (1-\epsilon)^{-n}\Upsilon(\psi(x_{1}\cdots x_{n}))
$$
In addition for $\nu_{\Upsilon}$ almost every $x\in\Sigma^+$ there is $N$ with
\begin{equation}\label{equation:sphericaldecay}
n\ge N \; \implies \; \Upsilon(\underline{\psi}(x_{1}\cdots x_{n})) \ge n^{-2}(\gamma(\phi\ast_\rho f))^n.
\end{equation}
\end{lemma}

\begin{proof}[Proof of Lemma \ref{lemma:sphericaldecay}]
Let $f,v\in\mathcal{H}_+$ be arbitrary. For brevity write $\gamma=\gamma(\phi\ast_\rho f)$.
Set 
$$
E_n= \left\{ x\in\Sigma  : \langle \rho(\underline{\psi}(x_{1}\cdots x_{n}))f,v\rangle_{\mathcal{H}} \ge \Upsilon(\underline{\psi}(x_{1}\cdots x_{n}))(1-\epsilon)^{-n} \right\},
$$
$$
\limsup E_n =  \left\{ x  : \langle \rho(\underline{\psi}(x_{1}\cdots x_{n}))f,v\rangle_{\mathcal{H}} \ge \Upsilon(\underline{\psi}(x_{1}\cdots x_{n}))(1-\epsilon)^{-n},\, n \;\mathrm{i.o.} \right\}.
$$ 
Using transitivity we can upgrade (\ref{eq:conformal1}) to
$$
\nu_{\Upsilon}([u])\le DR_n(uy)\Upsilon(\underline{\psi}(u)) \gamma^{-n}
$$
for any $u$ of length $n$ and $y\in\sigma^n [u]$, with constant $D$ independent of these.
 We compute 
\begin{align*}
\nu_{\Upsilon}(E_n) &= \sum_{A,a\in\mathcal{W}_1}\sum_{\substack{w\in\mathcal{W}_n^{A,a}:\\ \langle \rho(\underline{\psi}(w_{1}\cdots w_n))v,f\rangle_{\mathcal{H}} \ge \Upsilon(\underline{\psi}(w)) (1-\epsilon)^{-n}}} \nu_{\Upsilon}([w_1\cdots w_n])
\\
&\le  \sum_{A,a\in\mathcal{W}_1}\sum_{\substack{w\in\mathcal{W}_n^{A,a} :\\ \langle \rho(\underline{\psi}(w))f,v\rangle_{\mathcal{H}} \ge \Upsilon_{c}(\underline{\psi}(w))(1-\epsilon)^{-n}}} \frac{\nu_{\Upsilon}([w])}{\Upsilon(\underline{\psi}(w))}(1-\epsilon)^{n}\langle \rho(\underline{\psi}(w))f,v\rangle_{\mathcal{H}}
\\
&\le D \sum_{A,a\in\mathcal{W}_1}\sum_{w\mathcal{W}_n^{A,a}} \gamma^{-n}(1-\epsilon)^n R_n(w x)\langle \rho(\underline{\psi}(w))f,v\rangle_{\mathcal{H}}
\\
&\le  D\sum_{A,a\in\mathcal{W}_1} \langle  \phi^{A,a}_n[\gamma(1-\epsilon)^{-1}]\ast_\rho f,v\rangle_{\mathcal{H}}.
\end{align*}
Since $\gamma(1-\epsilon)^{-1}> \gamma(\phi\ast_\rho f)$ the convergence case of the Borel--Cantelli Lemma follows.

For the second part
set 
$$
E_n= \left\{ x\in\Sigma^+  : \Upsilon(\underline{\psi}(x_1\cdots x_n))\le \gamma^n n^{-2} \right\}.
$$
We have
\begin{align*}
\nu_{\Upsilon}(E_n) &= \sum_{A,a\in\mathcal{W}_1}\sum_{\substack{w\in\mathcal{W}_n^{A,a}:\\ \Upsilon(\underline{\psi}(w)) \le n^{-2}\gamma^n}} \nu_{\Upsilon}([w])
\\
&\le  \sum_{A,a\in\mathcal{W}_1}\sum_{\substack{w\in\mathcal{W}_n^{A,a} :\\ \Upsilon(\underline{\psi}(w)) \le n^{-2}\gamma^n}} C R_n(wy) \Upsilon(\underline{\psi}(w))\gamma^{-n}
\\
&\le D \sum_{A,a\in\mathcal{W}_1}\sum_{w\mathcal{W}_n^{A,a}} R_n(w x)\frac{1}{n^2}.
\end{align*}
The series $\sum_{n=1}^\infty n^{-2}$ is summable and so we conclude the convergence case of the Borel--Cantelli Lemma.
\end{proof}

We now specialize to the case of a random walk on $G$ as in Example \ref{example:rw}.
The structure of a random walk is useful as (for $t>\phi\ast \delta_e$),
\begin{equation}\label{eq:rw}
\phi^{B,b}[t](g)= \sum_{n=2}^\infty p(B)p^{\ast^{n-2}}(\psi^{-1}(B)g\psi^{-1}(b))p(b) t^{-n}.
\end{equation}
We further specialize to $G=F_{a,b}$, the free group with free basis $S=\left\{ a,a^{-1},b,b^{-1}\right\}$. We use the free basis to define a metric on $F_{a,b}$, writing $|g|_{F_{a,b}}$ for the word length of $g$ in the free basis. We also consider the random walk driven by the uniform probability on $S$, namely $p(s)=4^{-1}$ for each  $s\in S$. One can determine the radius of convergence in this case. 
\begin{proposition}[Theorem 3 in \cite{Kesten2}; Corollary 3.6.\ in \cite{woess}]
For $S^\mathbb{N},\psi,p$ given by the simple random walk on the free group $G=F_{a,b}$ we have 
$$
\gamma(\delta_e) = \frac{\sqrt{3}}{2}.
$$
\end{proposition}

There is a special class of functions $\Psi:F_{a,b}\to\mathbb{R}$ called \emph{the spherical functions} --- the reference we use for this discussion is the exposition of Fig\`{a}-Talamanca and Picardello \cite{FT}. By definition a spherical function $\Psi:F_{a,b}\to\mathbb{R}$ is constant on every sphere (is radial) and is multiplicative on the convolutional algebra of radial functions (and one tends to normalise so that $\Psi(e)=1$). Alternatively, the spherical functions are precisely eigenfunctions $p\ast \Psi = \hat{\gamma}(\Psi) \Psi$ with $p$ the uniform probability on the free basis, and can be characterized in terms the Poisson kernel. The function $\Psi(g)=\left( 1 + \frac{|g|_{F_{a,b}}}{2}\right) 3^{-|g|_{F_{a,b}}/2}$ uniquely (up to scaling) satisfies has $p\ast \Psi=2^{-1}3^{1/2} \Psi$. (The function $\Psi$ also arises as a matrix coefficient for the boundary representation of the free group.)

Our construction of $\phi_c[t]=\phi_c^{A,a}[t]$ and $\Upsilon=\Upsilon^{A,a;A,a}_c$ is not optimized to be radial, but we claim that there is a natural linear combination that is spherical for the free group. 
From (\ref{eq:rw}) we have
$$
\sum_{B,b\in\mathcal{W}_1}\phi^{B,b}[t](g)= \sum_{n=2}^\infty p^{\ast^{n}}(g)) t^{-n}.
$$
Pick slowly varying $c$ according to some $A,a$ and set $F[t] = \|\phi_c^{A,a}[t]\|_{\ell^2(G)}$. We have that $\|\phi_c^{B,b}[t]\|_{\ell^2(G)}/F(t)$ is bounded uniformly in $t>\gamma$ for each $B,b\in\mathcal{W}_1$. Set
\begin{align*}
\overline{\Upsilon}(g) &= \lim_{k\to\infty}\frac{\sum_{b,B;d,D}\langle \delta_g\ast \phi^{B,b}[t_k], \phi^{D,d}[t_k]\rangle_{\ell^2(G)}}{F(t)} 
\\
&= \frac{\langle \delta_g\ast \sum_{n=2}^\infty c_np^{\ast^{n}}(h))\delta_h t_k^{-n}, \sum_{n=2}^\infty c_n p^{\ast^{n}}(h)) t_k^{-n}\delta_h\rangle_{\ell^2(G)}}{F(t_k)}.
\end{align*}
Then $\overline{\Upsilon}(g)$ is constant on spheres since
$$
\#\left\{ w_n\in S^n, u_m\in S^m : gw_n = u_m\right\} = \#\left\{ \hat{w}_n\in S^n, \hat{u}_m\in S^m : \hat{g}\hat{w}_n = \hat{u}_m\right\}, 
$$
for each $g,\hat{g}$ with $|g|_{F_{a,b}}=|\hat{g}|_{F_{a,b}}$.
It is easily seen that 
$$
p\ast \overline{\Upsilon} = \lim_{k\to\infty}\frac{t_k \langle \ast \sum_{n=2}^\infty c_np^{\ast^{n+1}}(h)) t_k^{-(n+1)}\delta_h, \sum_{n=2}^\infty c_n p^{\ast^{n}}(h)) t_k^{-n}\delta_h\rangle_{\ell^2(G)}}{F(t_k)},
$$
and using the divergence of $F(t_k)$ we have $p\ast \overline{\Upsilon} = \gamma\overline{\Upsilon}$. We conclude the following, with $\Upsilon=\Upsilon^{A,a}$.
\begin{lemma}\label{lemma:spherical}
There is $C>0$ with
$$
C^{-1} \left( 1 + \frac{|g|_{F_{a,b}}}{2}\right) 3^{-|g|_{F_{a,b}}/2}\le \Upsilon(g) \le C \left( 1 + \frac{|g|_{F_{a,b}}}{2}\right) 3^{-|g|_{F_{a,b}}/2}.
$$
\end{lemma}

Lemma \ref{lemma:spherical} has consequences for the lower and upper bound on rate of decay. 
\begin{corollary}
For each $\epsilon>0$, a $\nu_{\Upsilon}$ typical point has
$$
\langle \rho(\underline{\psi}(x_{0}\cdots x_{n-1}))f,v\rangle \le \frac{|\underline{\psi}(x_{0}\cdots x_{n-1})|}{\sqrt{3}^{|\underline{\psi}(x_{0}\cdots x_{n-1})|}}\frac{(1+\epsilon)^n\sqrt{3}^n}{2^n}
$$
for $n$ sufficiently large (depending on $x$);
and a $\mu$ typical point has 
$$
\langle \rho(\underline{\psi}(x_{0}\cdots x_{n-1}))f,v\rangle \le \frac{(1+\epsilon)^n\sqrt{3}^{n}}{2^{n}}
$$
for $n$ sufficiently large (depending on $x$).
\end{corollary}

It is well-known that the drift $l(p)$ for the simple random walk $p$ on the free group on $2$ generators is equal to $1/2$ (using that $l(p) = \int b(g,\xi)dp(g)d\nu(\xi)$ with $b$ the Busemann additive cocycle --- see Proposition 2.2.\ of \cite{GMM}.). That says that for a $\mu$ typical $x$ we have 
$$
\lim_{N\to\infty}\frac{1}{n}|\underline{\psi}(x_{0}\cdots x_{n-1})|_{F_{a,b}}=\frac{1}{2}.
$$
This gives in particular a lower bound on the amount of cancellation in $\underline{\psi}(x_{0}\cdots x_{n-1})$. For a $\nu_{\Upsilon}$ typical point we have quantitatively more cancellation along a subsequence.
\begin{lemma}
A $\nu_{\Upsilon}$ typical point $x$ has
$$
\limsup_{n\to\infty}\frac{1}{n}|\underline{\psi}(x_{0}\cdots x_{n-1})|_{F_{a,b}}\le \frac{2\log 2- \log 3}{\log 3}< \frac{1}{2}.
$$

\end{lemma}
\begin{proof}
From Lemma \ref{lemma:sphericaldecay} together with Lemma \ref{lemma:spherical}, a $\nu_{\Upsilon}$ typical $x$ has
$$
C \left( 1 + \frac{|g|_{F_{a,b}}}{2}\right) 3^{-|g|_{F_{a,b}}/2}\le \Upsilon(\underline{\psi}(x_{0}\cdots x_{n-1}))\ge \sqrt{3}^{n}2^{-n}n^{-2}
$$
for all $n>N_x$.
Writing $g=\underline{\psi}(x_{0}\cdots x_{n-1})$, this gives for every $\epsilon>0$ and $N$ sufficiently large
$$
(\sqrt{3})^{-(n+|g|_{F_{a,b}})}\ge (1+\epsilon)^n2^{-n}
$$
Equivalently
$$
\frac{|g|_{F_{a,b}}}{2}\log 3+ \frac{n}{2}\log 3 \le n\log 2+  n\log (1+\epsilon).
$$

Let $\epsilon_k$ be a sequence with $\epsilon_k\to 0$ as $k\to \infty$. Then we have a full measure set $E$ for which
$$
\limsup_{n\to\infty}\frac{1}{n}|\underline{\psi}(x_{0}\cdots x_{n-1})|_{F_{a,b}}\le \frac{2\log 2- \log 3}{\log 3} + \frac{\log (1+\epsilon_k)}{\log 3}.
$$
for every $k$. The conclusion follows upon checking that 
$$
\frac{2\log 2- \log 3}{\log 3}< \frac{1}{2}.
$$
\end{proof}

\end{document}